\documentclass[11pt]{amsart}

 \usepackage{graphicx}
 \usepackage{amscd, color}

\usepackage{marginnote}

\swapnumbers

 \def\al{\alpha}
 \def\be{\beta}
  
 \def\de{\delta}

 \def\ga{\gamma}
  \def\G{\Gamma}

 \def\la{\lambda}

 \def\si{\sigma}
 \def\om{\omega}
 
 \def\d{\mathrm d }
 \def\EE{{\mathbf E}}
 \def\VV{{\mathbf V}}
 \def\XX {{\mathbf X}}
 \def\EN{{\mathcal E}}

 \def\HN {{\mathcal H}}
 
  \def\G{\Gamma}

 \def\R{{\mathbb R}}
 \def\Z{{\mathbb Z}}

  \def\A{{\mathcal  A}}

 \def\ov{\overline}

 \def\oo{\mathrm o}
 \def\tt{{\mathrm t}}

 \def\cH{H^\checkmark}

\def\tto{\longrightarrow}

  \renewcommand{\proofname}{{\bf Proof:}}

 \theoremstyle{plain}

  \swapnumbers

 \newtheorem{Thm}{Theorem}[section]
 
 \newtheorem{Lemma}[Thm]{\bf Lemma}
 
 \newtheorem{Corollary}[Thm]{\bf Corollary}
 \newtheorem{Theorem}[Thm]{\bf Theorem}
 \newtheorem{Proposition}[Thm]{\bf Proposition}

 \theoremstyle{definition}
 \newtheorem{Definition}[Thm]{\bf Definition}

 \theoremstyle{remark}
 \newtheorem{Remark}[Thm]{\bf Remark}

 \newtheoremstyle{Cl}% name
  {5pt}%      Space above
  {3pt}%      Space below
  {\sl}%   Body font
  {}%         Indent amount (empty = no indent, \parindent = para indent)
  {\it}% Thm head font
  {:}%        Punctuation after thm head
  {.5em}%     Space after thm head: " " = normal interword space;
  {}%         Thm head spec (can be left empty, meaning `normal')

 \theoremstyle{Cl}

 \def\begincproof{
                  \renewcommand{\proofname}{\it Proof:}
                  \begin{proof}
                 }

 \def\endcproof{
                \renewcommand{\qedsymbol}{$\diamondsuit$}
                \end{proof}
                \renewcommand{\qedsymbol}{\openbox}
                \renewcommand{\proofname}{\bf Proof:}
               }

 \renewcommand{\proofname}{{\bf Proof:}}

\title{Global results for Eikonal Hamilton-Jacobi equations on Networks}

\author{Antonio Siconolfi}
\address{Dipartimento di Matematica, Sapienza Universit\`a  di Roma, Italy.}
\email{siconolfi@mat.uniroma1.it}

\author{Alfonso Sorrentino}
\address{Dipartimento di Matematica, Universit\`a degli Studi di Roma ``Tor Vergata'', Rome, Italy.}
\email{sorrentino@mat.uniroma2.it}

\begin{document}
\maketitle

\begin{abstract}
We study a  one--parameter family of Eikonal  Hamilton-Jacobi
equations on an embedded network, and prove that there exists a
unique critical value for which  the corresponding equation admits
global solutions, in a suitable viscosity sense.  Such a solution is
identified, via an Hopf--Lax type formula, once an admissible trace
is assigned on an {\it intrinsic boundary}.   The salient point of
our method is to associate to the network an {\it abstract graph},
encoding all of the information on the complexity of the network,
and  to relate the differential equation  to a {\it discrete functional
equation} on the graph. Comparison principles and representation
formulae are proven in the supercritical case as well.
\end{abstract}

\section{Introduction}

{Over the last years there has been an increasing interest in the
study of the Hamilton-Jacobi Equation on networks and related
questions. These problems, in fact, involve a number of subtle
theoretical issues and have a great impact in the applications in
various fields, for example to data transmission, traffic management
problems, etc... While locally -- {\it i.e.}, on each branch of the
network ({\it arcs}) --, the study reduces to the analysis of
$1$-dimensional problems, the main difficulties arise in matching
together the information ``converging''  at the {\it juncture} of
two or more arcs, and relating the {\it local}
analysis at a juncture  with the {\it global} structure/topology of the network.\\

In this article, we provide a thorough discussion of the above issues in the case of Eikonal type Hamilton-Jacobi equations on
 embedded networks (in $\R^n$ or on a Riemannian manifold, see Remark
 \ref{Riemannian}). We show that there exists a unique  critical value
 for which  the corresponding equation admits global solutions,  and
 extend  most of the results known in the continuous  setting for the critical and  supercritical case. \\

The main rationale behind our approach consists in neatly
distinguishing between the {local} problem on the arcs and the
{global} analysis on the network. While the former can be solved by
means of (classical) $1$-dimensional viscosity techniques, the
latter is definitely more engaging.

Our novel idea is to tackle it
by associating to the network an {\it abstract graph}, encoding all
of the  information on the complexity of the network, and to relate the problem to a {\it discrete functional equation} on the graph.
This allows us
to pursue a global analysis of the equation -- that goes beyond what
happens at a  single juncture --, as well as
to prove uniqueness and comparison principles in a simpler way.
To the best of our knowledge, this is the first time that comparison type results are obtained in the network setting  by completely bypassing the difficulties involved in
the Crandall-Lions doubling variable method, in favor of a more direct analysis of a discrete equation.\\

In addition to this, by exploiting the simple geometry of the
abstract graph we are able to identify an intrinsic boundary -- the
{\it Aubry set} -- on which admissible traces can be assigned in
order to get unique critical solutions on the whole network; these
solutions can be represented by means of Hopf--Lax type formulae. In
the supercritical case  we  get existence and uniqueness of
solutions,  on  any open subset of the network, continuously
extending admissible data prescribed on the complement.

Let us point out that the problem of  formulating boundary problems
on the network and accordingly determining ``natural'' subsets on
which to assign boundary data is a subtle issue, yet not well
settled in the literature; we believe that our approach helps
clarify this matter, at least in the class of equations that we are considering.\\

The notions of viscosity solution and subsolution that we adopt are
very natural in this setting (see Definitions \ref{defsolsubsol} and
\ref{defsolsubsol2}). More specifically, the tests we use at
vertices are classical in viscosity solutions theory and consist in
(unilateral) state constraint type boundary conditions, introduced
by Soner \cite{Soner1} to study control problems with constraints.
In this regard,  the notion of solution requires that at each vertex
the state constraint condition holds for at least one arc ending
there: it does not require other mixing conditions (on the vertices)
 between equations defined on different incident arcs.

Very recently,  the same notion of solution has been also considered
by Lions and Souganidis in \cite{LionsSouganidis}  to deal with  one
dimensional junction-type problems for non convex discounted
Hamilton-Jacobi equations and study its well-posedness ({\it i.e.},
comparison principle and existence). Global solutions on networks, however, are not therein studied.\\

As far as subsolutions are concerned, we only ask  that they are
continuous on the network and are (viscosity) subsolutions to the
equation on the interior of each arc: no extra conditions are
required on vertices. These assumptions are the minimal requirements
that one needs to ask and, at a first sight, it might seem
surprising that they are sufficient to develop a significant global
theory. However, the validity of this approach is supported, among
other things, by the fact that  the notion of solutions  can be
recovered in terms of maximal subsolution attaining a
specific value at a given point (vertex or internal point); see  Theorem \ref{deniro}.\\

We also wish to point out that our hypothesis both on the topology
of network and the Hamiltonians are very general. As far as the
network is concerned, we only ask it to be made up by finite arcs
and connected: hence,
it may well include multiple connections between different vertices, as well as the presence of loops. \\
The Hamiltonians  are assumed continuous in both variables,
quasiconvex and coercive in the first order variable on any arc.
Hamiltonians on different arcs  are independent one from
the others and no compatibility conditions at the vertices are required. See subsection \ref{secHam} for more details.\\

We are confident that this very same set of ideas
can be successfully applied to a broad range of  other problems: for example,  to the study of the
{\it discounted} Hamilton-Jacobi equation on networks or to prove
{\it homogenization} results for the Hamilton-Jacobi equation on periodic networks (also known as {\it topological crystals}).
We plan to address these and other questions  in a future work (in  preparation).

\subsection{Previous  related literature}
There is a huge amount of literature related to differential
equations on networks,  or others  non-regular geometric structures
(ramified/stratified spaces), in various contexts:
 hyperbolic problems, traffic flows, evolutionary equations, (regional) control problems, Hamilton-Jacobi equations, etc...
 An exhaustive description of the state of the art in all of these areas would go well beyond the aims of this paper;  just to mention a few noteworthy items:
\cite{AchdouCamilliCutriTchou, BarlesBrianiChasseigne1, BarlesBrianiChasseigne2, BressanHong, CamilliMarchi, CamilliMarchiSchieborn, DaviniFathiIturriagaZavidovique,
GaliseImbertMonneau, GaravelloPiccolo, ImbertMonneau1, ImbertMonneau2, ImbertMonneauZidani, LionsSouganidis,
 PokornyiBorovskikh, RaoSiconolfiZidani, CamilliSchieborn, Soner1}. See also references therein.\\

A model similar to ours has been previously considered by Camilli
and Schieborn in \cite{CamilliSchieborn}, however just in the
supercritical case and under some restriction on the topology of the
network. In comparison with their hypothesis, we do not require
continuity of the Hamiltonians at the vertices (and accordingly,  no
mixed conditions on the test functions at the vertices)
and we do not ask  a-priori existence of a regular strict subsolution.\\

Other relevant recent contributions are \cite{LionsSouganidis} (that we have already mentioned above) and
\cite{ImbertMonneau1}. In particular, the latter is a substantial work  -- whose point of view and techniques are rather different from ours -- in which Imbert and Monneau
attempt to recover the doubling variable method to their setting, by introducing an extra parameter (the flux limiter), a companion equation (the {junction condition}) and
by using special vertex test functions.
See also other related  works by the same authors and collaborators \cite{GaliseImbertMonneau,  ImbertMonneauZidani, ImbertMonneau2}.\\

Our analysis of the discrete functional equation is based on ideas and
techniques inspired by the so-called {\it weak KAM theory},  firstly
developed by Fathi \cite{Fathi} for the study of Tonelli
Hamiltonian systems on closed manifolds (see also
\cite{Sorrentinobook}). Developing a similar approach in the discrete
setting is very natural and has been already exploited in several
other works. In \cite{BernardBuffoni, BernardBuffoni2}, for example,
a discretization of weak KAM theory was applied to investigate the
properties of  optimal transport maps; a  more systematic
development of a {discrete weak KAM theory} for {\it cost functions}
was described by Zavidovique in \cite{Zavidovique1,Zavidovique2} (see also \cite{DaviniFathiIturriagaZavidovique}). In particular, \cite{Zavidovique2} shares ideas similar to
ours, although our setting has the peculiarity of this interplay between the discrete structure and the embedded network.

From a more dynamical systems point of view, a discrete analogue of Aubry-Mather theory and weak KAM theory was also discussed in \cite{Gomes} (see \cite{SuThieullen} for a recent related work).\\

\medskip

\noindent {\bf Acknowledgments.} This work has been supported by the
INdAM-GNAMPA Research project :``{\it Fenomeni asintotici e
omogeneizzazione}''. A. Siconolfi acknowledges the Progetto Ateneo
2015-- Rome {\it La Sapienza} University: ``{\it Asintotica e
omogeneizzazione  di dinamiche Hamiltoniane}''.  A. Sorrentino
acknowledges the PRIN- 2012-74FYK7 grant: ``{\it Variational and
perturbative aspects of nonlinear differential problems}''.

\bigskip

%%%%%%%%%%%%%%

\section{Preliminaries on Graph Theory} \label{prelim}
In this section we  recall some basic material on the theory of
abstract graphs  and on functions defined on them. For a more
detailed presentation of these and other related topics, we refer
the interested readers, for instance, to \cite{Sunada}.

\subsection{Abstract graphs}

\smallskip
A (abstract) graph $\XX=(\VV ,\EE )$ is an ordered pair of  sets
$\VV$ and $\EE$, which are  called, respectively, {\it vertices} and
(directed) {\it edges}, plus two functions:
$$
\oo: \EE \longrightarrow \VV
$$
and
\begin{eqnarray*}
 \overline{\phantom{o}}: \EE &\longrightarrow& \EE \\
e &\longmapsto& \ov e,
\end{eqnarray*}
with the latter assumed to be a  fixed-point-free involution, namely
satisfying
\[\ov e \neq e \qquad\hbox{and} \qquad \ov{\ov e}= e \qquad\hbox{for any $e \in \EE$.}\]
We give the following geometric picture of the setting: $\oo(e)$  is
the {\it origin} (initial vertex) of $e$ and $\ov e$ its {\it
reversed} edge, namely the same edge but with the opposite
orientation. Analogously  we define
\[\tt  (e)= \oo (\ov e)\]
the {\it terminal} vertex of $e$. The following compatibility
condition holds true
\[\tt  (\ov e)= \oo (\ov{\ov e})= \oo (e).\]
We say that $e$ links $\oo (e)$ to $\tt  (e)$, observe that it might
well happen that $\oo(e)=\tt(e)$, and in this case  $e$ will be
called a {\it loop}.
 An
edge is  also said to be incident on $\oo (e)$ and $\tt  (e)$. Two
vertices are called {\it adjacent} if there is an edge linking them or, in
other terms, if there is an edge incident on both of them. \\
\smallskip

We say that the graph is {\it finite} if  the
set $\EE$, and consequently $\VV$, has a finite number of elements.
We denote by $|\VV|$, $|\EE|$ the number of
vertices and edges. \\

\smallskip

We define a {\it path} to be a finite sequence
of concatenated  edges, namely   $\xi=(e_1, \cdots, e_M)=(e_i)_{i=1}^M$ satisfying
\[\tt  (e_j)=\oo (e_{j+1}) \qquad\hbox{for any $j= 1, \cdots, M-1$.}\]
We set $\oo (\xi)= \oo (e_1)$, $\tt  (\xi)= \tt  (e_M)$, and call
them the initial and final vertex of the path. We  say that $\xi$
links $\oo (\xi)$ to $\tt  (\xi)$, we also say that $\xi$ is
{incident} on some vertex if there is some edge composing the path
incident on it.

Given two paths $\xi$, $\eta$, we say that  $\xi$ is contained in
$\eta$, mathematically  $\xi \subset \eta$, if the edges  of $\xi$
make up a subset of the  edges of $\eta$.  If such a subset is
proper, we say that $\xi$ is {\it properly contained} in $\eta$. If
$\tt(\xi) = \oo(\eta)$, we denote by $\xi \cup \eta$ the path
obtained via concatenation of $\xi$ and $\eta$.

We call a path  a {\it loop} or a {\it cycle} if $\oo (\xi)= \tt
(\xi)$. A  path without repetition of vertices except possibly the
initial and terminal ones will be called {\it simple}, in other
terms $\xi=(e_i)_1^M$ is simple if
\[\tt(e_i) = \tt(e_j) \, \Rightarrow i=j,\]
or if there are no cycles properly contained in $\xi$.  Note that
there are finitely many simple paths in a finite graph.

 \smallskip

A graph is called  {\it connected} if any two vertices are linked by
some path.
 All of the graphs we will consider hereafter are
understood to be connected and finite. \\

\smallskip

Given $x \in   \VV$, we set
\begin{equation}\label{defstar}
\EE_x= \{e \in \EE \mid \oo (e) =x\},
\end{equation}
which we call $\EE_x$ the {\it star centered at $x$}; it should be considered as a sort of
tangent space to the graph at $x$. The cardinality  of
$\EE_x$ is called the {\it degree} (or {\it valence}) of the vertex
$x$.

\subsection{Functions  on  graphs} In the following we will be interested in functions defined on abstract graphs.
It is useful to introduce the following notions. \\

We define:
\begin{itemize}
\item[-] the  {\it $0$--cochain group} $C^0( \XX, \R)$  as the
space of functions from $\VV$ to $\R$. This space play the role of functions on the graph.
\item[-] The {\it $1$--cochain group} $C^1( \XX, \R)$ as the
space of functions from $\EE$ to $\R$, the compatibility condition
$\om(\ov e)= - \om(e).$ This space plays the role of $1$-forms on
the graph. From now  on we will indicate the reverse edge $\ov e$ by
$-e$ and  we will consider the pairing $\langle \om,e \rangle :=
\om(e)$.
\end{itemize}

\smallskip

The relation between  $C^0( \XX,\Z)$ and $C^1( \XX, \Z)$ can be
expressed in terms of the so-called {\it coboundary operator}, or
{\it differential}, $\d: C^0( \XX,\Z) \to C^1( \XX, \Z)$, which is
defined for any $f\in C^0( \XX,\Z)$ and $e\in \EE$ as
\[\d f (e) :=   f(\tt(e))- f(\oo(e)).\\ \]

\bigskip

We can embed these spaces with the standard topology.
A notion of
convergence on the cochain spaces is given via
\begin{eqnarray*}
f_n \longrightarrow f &\Longleftrightarrow& f_n(x) \longrightarrow f(x) \qquad \hbox{for any $x \in \VV$} \\
 \om_n \longrightarrow \om &\Longleftrightarrow& \om_n(e) \longrightarrow \om(e) \qquad \hbox{for any $e \in \EE$.}\\
\end{eqnarray*}

\smallskip

A sequence $f_n$ is said {\it equibounded} if
\[|f_n(x)| \leq \be \qquad\hbox{for any $x \in \VV$, some $\be
>0$;}\]
similarly $\om_n$ is said  equibounded if
\[|\langle \om_n,e \rangle | \leq \be \qquad\hbox{for any $e \in \EE$, some $\be
>0$.}\]
It is clear that any equibounded sequences $f_n$, $\om_n$ are
convergent, up to subsequences.\\

We directly deduce from the above definitions:

\smallskip

\begin{Proposition} \label{equi} Let $f_n$, $f$ be in $C^0(\XX,\R)$
\begin{itemize}
    \item [{i)}] if $f_n \tto f$, then $\d f_n \tto \d f$;
    \item [{ii)}] if $\d f_n$ is equibounded and the sequence
    $f_n(x_0)$ is bounded for some vertex $x_0$, then $f_n$ is
    convergent, up to subsequences.
\end{itemize}
\end{Proposition}

\medskip

\section{Setting}\label{setting}
In this section we first explain our setting,
 namely what is an {\it embedded network} and what we mean by {\it Hamiltonian} on a network.
 Then we introduce the class of {\it Hamilton-Jacobi equations} on a network we are interested in, and specify the notions of solutions and
 subsolutions.

\subsection{Embedded networks}\label{networks}

An {\it embedded network}, or {\it continuous graph}, is a subset $
\Gamma \subset \R^N$ of the form
\[ \Gamma = \bigcup_{\ga \in \EN} \, \gamma([0,1]) \subset \R^N,\]
where $\EN$ is a finite collection of regular simple oriented
curves, called {\it arcs} of the network,  that we assume, without
any loss of generality, parameterized on $[0,1]$.
We denote by $\EN^*$ the subset of arcs $\ga$ which are closed, namely with $\ga(0)=\ga(1)$. \\

\begin{Remark}\label{Riemannian}
Our setting can be easily extended to the case in which $\Gamma$
is embedded in a Riemannian manifold  $(M,g)$, for example by means of Nash embedding theorem \cite{Nash}.\\
\end{Remark}

Observe that on the support of any arc $\ga$, we  also consider the
inverse parametrization   defined as
\[\widetilde \ga(s)= \ga( 1 -s) \qquad\hbox{for $s \in [0,1]$.}\]
We call $\widetilde \ga$ the {\it inverse arc} of $ \ga$.  We assume
\begin{equation}\label{netw}
    \ga((0,1)) \cap \ga'((0,1)) = \emptyset \qquad\hbox{whenever $\ga \neq
\ga'$, $\ga \neq \widetilde{\ga'}$.}\\
\end{equation}

\medskip
 We call  {\it vertices} the
initial  and terminal points of the arcs, and denote  by  $\VV$ the
sets of all such vertices. Note that \eqref{netw} implies that
\[\ga((0,1)) \cap \VV  = \emptyset \qquad\hbox{for any $\ga \in
\EN$.}\] We assume that the network  is  connected, namely given two
vertices there is a finite concatenation of  arcs linking them.\\

The network $\Gamma$ inherits a {\it geodesic distance}, denoted by
$d_\Gamma$,  from the Euclidean metric of $\R^N$. Hence, hereafter
the notions of continuity and Lipschitz continuity,  when referred
to functions defined on $\Gamma$,  must be understood with respect
to such distance and the induced topology.\\

\smallskip

We can also consider a {\it differential structure} on $\Gamma$ by
defining  the tangent space at any $x \in \G \setminus \VV$ as
\begin{eqnarray*}
  T_\Gamma(x) &=&\{ \lambda \, \dot\ga(t) \mid \lambda \in \R, \,\ga \in \EN, \, t \in (0,1) \;\hbox{and} \; x=\ga(t) \}
\end{eqnarray*}
and  the cotangent space $T^*_\Gamma(x)$ as the dual space
$(T_\G(x))^*$;  namely,   it is the set of linear functionals $p: T_\G(x) \longrightarrow \R$.}\\

We will say that  a function $f: \Gamma \to \R$ is of class $C^1(\G
\setminus V)$ if it is continuous  in $\G$ and
\[ t \mapsto f(\ga(t)) \;\;\hbox{is of class $C^1$ in $(0,1)$ for any $\ga \in
\EN$.}\] For such a function we define $D_\G f(x)$, where $x =
\ga(t_0)$ for some $\ga \in \EN$ and $t_0 \in (0,1)$, as the unique
covector in $T^*_\Gamma(x)$ satisfying

\[ ( D_\G f(x) , \dot\ga(t_0) ) = \frac d{dt} f(\ga(t))\big |_{t=t_0},\]
where $( \cdot,\cdot )$ denotes the  pairing between covectors and
vectors.

Notice that this definition is invariant for a change of
parametrization from $\ga$ to $\widetilde \ga$.\\

\smallskip

We can associate to any continuous network $\Gamma$ an abstract
graph $\XX= (\VV, \EE)$  with the same vertices of the
network
 and edges corresponding to  the arcs.
More precisely, we consider an  abstract set $\EE$ with a bijection
\begin{equation}\label{defPsi}
  \Psi: \EE \longrightarrow \EN.
  \end{equation}
This induces  maps  $o : \EE \longrightarrow \VV$,
$\overline{\phantom{o}}: \EE \longrightarrow \EE $
 via
 \begin{eqnarray*}
   \oo(e) =  \Psi(e)(0)  \quad {\rm and} \quad
   \overline e = \Psi^{-1}(\widetilde{\Psi(e)}),
 \end{eqnarray*}
satisfying the properties in the definition of graph. Intuitively,
in the passage from the embedded network to the underlying abstract
graph  $\XX$,  the arcs become  {\it immaterial} edges.\\

\subsection{Hamiltonians on networks}\label{secHam}
A Hamiltonian on a network $\Gamma$ is a collection of Hamiltonians
$\HN=\{H_{\gamma}\}_{\ga \in \EN}$, where
\begin{eqnarray*}
H_\ga: [0,1] \times \R &\longrightarrow& \R\\
(s,p) &\longmapsto& H_\ga(s,p)
\end{eqnarray*}
satisfies
\begin{equation}\label{ovgamma}
   H_{\widetilde\ga}(s,p) = H_{\ga}(1-s,-p) \qquad\hbox{for any $\ga
\in \EN$} \\
\end{equation}

\medskip

Notice that we are not assuming any periodicity on
$H_\ga$ when $\ga$ is a closed curve.\\

\smallskip

We  require  any  $H_\ga$ to be:
\begin{itemize}
    \item[{\bf (H$\ga$1)}] continuous in $(s,p)$;
    \item[{\bf (H$\ga$2)}] coercive in $p$;
    \item[{\bf (H$\ga$3)}]  quasiconvex in $p$, with
    \[ \mathrm{Int} \,\big(\{p \mid H_\ga(x,p) \leq a\} \big)=  \{p \mid H_\ga(x,p) < a\} \quad\hbox{for any $a\in \R$,} \]
    where $\mathrm{Int}\big(\cdot \big)$ denotes the interior of a set.\\
\end{itemize}

\smallskip

We point out that, throughout the paper, the term (sub)solution to
Hamilton--Jacobi equations involving the $H_\ga$'s, must be
understood in the viscosity sense, see for example \cite{BardiCapuzzo, Barles} for a comprehensive treatment of viscosity
solutions theory.\\

\smallskip
We  set for any $\ga \in \EN$
\begin{eqnarray}
  a_\ga &:=&   \max_{s\in[0,1]}\min_{p \in \R} H_\ga(s,p)       \label{a00}  \\
  c_\ga &:=& \min \{a \, :\, H_\ga=a \;\hbox{admits periodic subsolutions}\}. \label{c00}
\end{eqnarray}

By periodic subsolution, we  mean subsolution to the equation in
$(0,1)$ taking the same value at the endpoints.\\
\smallskip

\begin{Remark}\label{cga} The definition of $c_\ga$ is indeed well-posed. In
fact, given $\ga \in \EN$, because of the the compactness of
$[0,1]$, we can choose $a$  large enough to have
\[H(s,0) \leq a \qquad\hbox{any $s \in (0,1)$.}\]
This shows that any constant function  is a subsolution and,
consequently, the set in the definition of $c_\ga$ is non-empty. It
is also bounded from below since for $a < a_\ga$ the corresponding
equation does not admit subsolutions and, therefore,  it does not admit
periodic ones. Finally, by basic stability properties in
viscosity solution theory, there  exists a periodic subsolution at
the level $c_\ga$, which justifies the minimum appearing  in the
definition.\\
We will essentially use $c_\ga$ for $\ga \in \EN^*$, but  in
principle the definition and the above considerations hold for any
$\ga$.\\
\end{Remark}

We stress that
\[a_\ga \leq c_\ga \qquad\hbox{for any   $\ga \in \EN$}. \]
\smallskip

We further define
\begin{equation}\label{a0}
 a_0 := \max \left \{\max_{\ga \in \EN \setminus \EN^*} a_\ga, \max_{ \ga \in  \EN^*} c_\ga \right \}.
\end{equation}
\smallskip

We require a further condition:

\smallskip

\begin{itemize}
    \item[{\bf (H$\ga$4)}]  given any $\ga \in \EN$ with $a_\ga =a_0$,  the map  $s \longmapsto  \min_{p \in \R} H_\ga(s,p)$
    is constant  in $[0,1]$.\\
\end{itemize}

\begin{Remark}\label{trump}    The
main role of {\bf (H$\ga$4)} is to ensure uniqueness of solutions to
the Dirichlet problem associated to the equation $H_\ga= a_\ga$, at
least for the $\ga$'s with $a_\ga=a_0$. The uniqueness property for
such kind of problems holds in general when the equation admits a
strict subsolution, which is not the case at the level $a_\ga$. The
relevant consequence of  condition {\bf (H$\ga$4)} is that the
family of subsolutions to $H_\ga= a_\ga$ reduces to a singleton, up
to additive constants, see Proposition \ref{equnique}.

Notice finally that  condition {\bf (H$\ga$4)} is automatically
satisfied if the $H_\ga$'s are independent of the state variable.
\end{Remark}
\smallskip

\medskip

\subsection{The Eikonal Hamilton--Jacobi equation on networks}\label{secHJeq}

\medskip

We define a notion of subsolution and solution to  an  equation of the form
\begin{equation} \label{HJ}\tag{$\mathcal{H}J\mathit{a}$}
\HN(x,Du)= a  \qquad\hbox{on $\Gamma$.}
\end{equation}
where $a \in \R$. This notation synthetically  indicates  the family
(for $\ga$ varying in $\EN$) of Hamilton--Jacobi equations
\begin{equation} \label{HJg}\tag{$HJ_\ga \mathit{a}$}
H_\ga(s,(u \circ \ga)')= a  \qquad\hbox{on $(0,1)$.} \\
\end{equation}

\bigskip

 We start by recalling some terminology of viscosity solutions
 theory.\\

\begin{Definition}
Given a continuous function $w$ in $[0,1]$ and a function $\varphi
\in C^1([0,1]$, we say that:
\begin{itemize}
\item[-] $\varphi$ is {\it supertangent} to $w$ at $s \in (0,1)$ if
\[ w = \varphi \;\;\hbox{at $s$} \quad {\rm and}\quad w \geq \varphi \;\;\hbox{in $(s-\de, s+ \de)$ for some $\de >0$.}
\]
The notion of {\it subtangent} is given by just replacing $\geq$  by
$\leq$ in the above formula.
\item[-] $\varphi$ is a {\it constrained subtangent} to $w$ at
$1$ if
\[ w = \varphi \;\;\hbox{at $1$} \quad {\rm and}\quad w \geq \varphi \;\;\hbox{in $(1-\de, 1)$ for some $\de >0$.}
\]
A similar notion, with obvious adaptations, can be given at $t= 0$.\\
\end{itemize}
\end{Definition}

%\smallskip

\begin{Definition}\label{boundary} Given  a continuous function $w$ in $[0,1]$, a point $s_0 \in \{0,1\}$,
 we say that it satisfies
{\it the state constraint boundary condition} for \eqref{HJg} at
$s_0$ if
 \[ H_\ga (s_0, \varphi'(s_0)) \geq a\]
for any $\varphi$ that is a constrained $C^1$ subtangent  to $w$ at $s_0$.
\end{Definition}

\smallskip

\begin{Definition} \label{defsolsubsol} We say that $u: \Gamma \longrightarrow \R$ is {\it subsolution} to \eqref{HJ}
if
\begin{itemize}
    \item [{i)}] it is continuous on
    $\Gamma$;
    \item [{ii)}] $s \mapsto u(\ga(s))$ is subsolution to
    \eqref{HJg} in $(0,1)$ for any $\ga \in \EN$.
\end{itemize}
\smallskip
We say that $u$ is {\it solution } to \eqref{HJ} if
\begin{itemize}
     \item [{i)}] it is  continuous;
     \item [{ii)}] $s \mapsto u(\ga(s))$   is solution of
    \eqref{HJg} in $(0,1)$ for any $\ga \in \EN$;
    \item [{iii)}] for
    every vertex $x$ there is at least one arc $\ga$, having $x$ as terminal point, such that   $ u(\ga(s))$
    satisfies the state constraint boundary   condition for \eqref{HJg}  at $s=1$.
\end{itemize}
\end{Definition}

Compare also this definition with the one in \cite{LionsSouganidis}. As far as we know, the idea of imposing a supersolution condition on just one arc incident to a given vertex, first appeared in \cite{CamilliSchieborn}.

We do not provide a notion of supersolution. This could be done
straightforwardly  but we will not need it in the remainder of the
paper.\\

\begin{Definition}\label{defsolsubsol2}
 Given  an open (in the relative topology) subset $\G' \subset \G$, we say that a continuous function $u:\G \to\R$ is {\it solution to
\eqref{HJ} in $\Gamma'$},  if  for any $x \in \G' \setminus \VV$, $x
=\ga(s_0)$ with $\ga \in \EN$, $s_0 \in (0,1)$, the usual viscosity
solution condition holds true for $u \circ \ga$ at $s_0$. If instead
$x \in \G' \cap \VV$,  we  require condition {iii)} in Definition
\ref{defsolsubsol} to
hold.
\end{Definition}

\smallskip

\begin{Remark}\label{cgabis}
The  definition of (sub)solutions on $\Gamma$ requires $u\circ \ga$
to be a (sub)solution of the corresponding equation in $(0,1)$ on
any arc $\ga$. If, in particular  $\ga$ is a closed curve, we must
have in addition $u(\ga(0))=u(\ga(1))$. This explains why on any arc
$\ga \in \EN^*$ we are solely interested in periodic (sub)solutions,
namely (sub) solutions in $(0,1)$ taking the same value at $0$ and
$1$. This also explains the role of $c_\ga$.
\end{Remark}

\smallskip

 Given a continuous function $u$
defined in $[0,1]$, it is apparent that a $C^1$ function  $\varphi$
is supertangent (resp. subtangent ) to $u$ at $s_0 \in (0,1)$ if and
only if $\widetilde \varphi(s):= \varphi(1-s)$ is supertangent
(resp. subtangent ) to $s \longmapsto u(1-s)$ at $1-s_0$. Taking into
account \eqref{ovgamma}, we derive the following result.

\begin{Proposition}\label{inverti} Given an arc $\ga$,  a function $u(s)$ is subsolution (resp. solution) to \eqref{HJg}
if and only if $s \mapsto u(1-s)$ is subsolution (resp. solution) to
the the same equation  with $H_{\widetilde \ga}$ in place of
$H_\ga$.
\end{Proposition}

\smallskip

It is not difficult  to see  that Lipschitz--continuity of
subsolutions on any arc, coming from the coercivity condition in
{\bf (H$\ga$2)},  implies Lipschitz--continuity in $\Gamma$ with
respect to the geodesic distance. We provide a proof in Appendix
\ref{proofs} for reader's convenience.

\begin{Proposition} \label{Lipcont}
The family of  subsolutions to \eqref{HJg}, provided it is not
empty, is equiLipschitz continuous on $\Gamma$ with respect to the
geodesic distance $d_\Gamma$.
\end{Proposition}

\smallskip
We derive from the previous result plus basic properties of
viscosity solutions the existence of the maximal subsolution
attaining a given value at a given point of the network.

\begin{Proposition}\label{predeniro} Let $a$ be such that the equation \eqref{HJ} admits subsolution in $\G$. Given $y \in \G$, $\al \in \R$, the function
\[w(x)= \max\{ u(x) \mid \;\hbox{subsolution to \eqref{HJ} with $u(y)
= \al$}\}\] is still a subsolution.
\end{Proposition}

\smallskip

\medskip

\section{Strategy of the proof}\label{strategy}
The remaining of the article consists in the proof of our results
on  existence,  uniqueness and  regularity  of global (sub)solutions
to the Eikonal Hamilton-Jacobi equation on $\G$. For the reader's
convenience, a summary of all of our main results will be detailed
in section \ref{summary}.

Before starting, we believe it might be useful to provide here an outline of the forthcoming discussion.\\

In section \ref{seclocal}, we focus on the {\it local} problem on
each arc of the network. Namely, for each  $\ga\in\EN $  we study
the existence of (sub)solutions to the $1$-dimensional Eikonal
Hamilton-Jacobi equation \eqref{HJg} with boundary conditions.  In
particular:
\begin{itemize}
\item[-] We show that under suitable {\it admissibility conditions} on the boundary data, see \eqref{compa}, there exists a unique solution and we provide a representation formula (Proposition \ref{statextra}).
\item[-]  We derive a characterization of condition {iii)} in Definition \ref{defsolsubsol} in terms of this representation formula (Proposition \ref{state}).
%\item[-] The Lipschitz continuity on $\G$, claimed in {\bf Main Theorem (i)}, follows from Proposition \ref{Lipcont}.
\end{itemize}

\smallskip

In section \ref{statgraph} we concentrate on the global aspects of
the problem.
\begin{itemize}
\item[-]  We introduce a {\it discrete functional equation}
\eqref{HJa} on the abstract graph  $\XX$ and provide the
corresponding notions of solutions and subsolutions. The crucial
result linking solutions to this equation and solutions to
\eqref{HJa} is proven in
Proposition \ref{relationHJDEF}. %; this proves {\bf Main Theorem (ii)}.
\item[-] In \eqref{defcritvalue} we define {\it Ma\~n\'e
critical value} $c(\HN)$. We first prove that this is the unique
value for which solutions to the discrete functional
 equation may exist (Proposition \ref{propcritvalue}), and then that  the critical equation ($\mathcal{D}\mathit{FEc}$)
   admits indeed solutions (Theorem \ref{critico}). %This part completes the proof of  {\bf Main Theorem (i)}.
\item[-] In \eqref{Aubry1} and  \eqref{Aubry2} we define
 the {\it Aubry set} $\A_\XX^*$ and the {\it projected Aubry set}
 $\A_\XX$, which  are non-empty (Proposition \ref{Aubrynonempty}).
We prove in Proposition \ref{uniquenessset} that $\A_\XX$ is a
uniqueness set
 % as claimed in {\bf Main Theorem (iii)},
 and provide a Hopf--Lax type representation formula for the solutions
  to ($\mathcal{D}\mathit{FEc}$) in terms of its values on $\A_\XX$.
\end{itemize}

\smallskip

The supercritical case will be discussed in parallel to the critical one (see Proposition  \ref{prop5.8}, Proposition \ref{prop5.11} and Theorem \ref{prop5.23}).\\

Finally, in section \ref{backlocal} we switch our attention back to
the immersed network:

\begin{itemize}
    \item [-] We prove in Theorem \ref{deniro} that the notion of solution can be
recovered in terms of maximal subsolution attaining a specific value
at a given point.
    \item[-] We introduce  the analogue of the Aubry set
on the network,  we show in Theorem \ref{regolauno} that all
critical subsolutions are of class $C^1$ on it and they all have the
same differential on this set.
\item[-] We show the existence of   $C^1$ critical subsolutions that are strict
outside of the Aubry set (Theorem \ref{regolaregola}).
\item[-] We  provide representation formulae  and uniqueness results with traces that are not necessarily defined on
vertices (Theorem \ref{thisistheend}).
\end{itemize}

\bigskip

\section{Local part: the Eikonal Hamilton-Jacobi equation with  boundary conditions on arcs} \label{seclocal}

In this section we focus on a single arc $\ga$ and study the family
of equations \eqref{HJg} in $(0,1)$ plus suitable boundary
conditions. We assume
 \[a \geq  a_0=\max \left \{\max_{\ga \in \EN \setminus \EN^*} a_\ga, \max_{ \ga \in  \EN^*} c_\ga \right
 \}.\]
Our aim is to find  admissible conditions on boundary data at $s=0$
and $s=1$ to get solutions of the corresponding Dirichlet problem,
to show uniqueness of such solutions  and, finally, to provide a
characterization of maximal subsolutions taking a given value at
$s=0$ via state constraint boundary conditions.

We need  specific results when $\ga$ is a closed curve  because in
this case we are solely interested to periodic (sub)solutions, as
explained in Remark \ref{cgabis}. We address the issue in Subsection
\ref{tre}. In the first subsections \ref{uno} and \ref{due} we will
not distinguish between $\ga$ closed or not, and provide an unified
presentation of the material.

The results are not new,  we write down nevertheless the
one--dimensional representation formulae, which  are  easy to handle
and allows a direct and simplified treatment of the matter. We
recall that, due to coercivity and quasiconvexity assumptions, all
subsolutions to \eqref{HJg} are Lipschitz--continuous in $[0,1]$,
and, in addition the notion of viscosity and a.e. subsolution are
equivalent. Also notice that the subsolution property is not
affected by addiction of constants.

To ease notation, we  write $H(s,p)$ instead of $H_\ga(s,p)$, and
accordingly we consider equation \eqref{HJg} with $H$ in place of
$H_\ga$.  Moreover, we   denote by $\cH$ the Hamiltonian
$H_{\widetilde{\ga}}$. We recall that the assumptions {\bf
(H$\ga$1)}--{\bf (H$\ga$4)} are in force.

\subsection{Setting of the local problem} \label{uno}
We  set for $s \in [0,1]$
\begin{eqnarray}
  \si^+_a(s) &=& \max \{p \mid H(s,p)=a\} \label{sia}\\
  \si^-_a(s) &=& \min \{p \mid H(s,p)=a\} \label{siabis}.
\end{eqnarray}
If $a > a_\ga$, we have by {\bf (HJ$\ga$3)}
\begin{equation}\label{siatris}
    (\si_a^-(s),\si_a^+(s)) = \{p \mid H(s,p) <a\}
    \qquad\hbox{for $s \in [0,1]$.}
\end{equation}
\smallskip

We deduce from assumption {\bf
 (H$\ga$4)} that if $a_\ga=a_0$
 \begin{equation}\label{siasia}
  \si^+_{a_\ga}(s)= \si^-_{a_\ga}(s) \qquad\hbox{for any $s \in
[0,1]$.}
\end{equation}

\smallskip

\begin{Proposition}\label{siac1} The functions $s \longmapsto \si^+_{a}(s)$, $s \longmapsto
\si^-_{a}(s)$ are continuous in $[0,1]$ for any $a \geq a_\ga$.
\end{Proposition}

\begin{proof}
It follows directly  from the continuity and the coercivity of $H$
that the function $s \longmapsto \si^+_{a_\ga}(s)=\si^-_{a_\ga}(s)$
is continuous. If $a > a_\ga$, the assertion  follows from the fact
that $\si^+_{a}(s)$, $\si^-_{a}(s)$  are univocally determined for
any $s$  by the conditions $H(s,\si^+_{a}(s))=H(s,\si^-_{a}(s))=a$
and, respectively, $\si^+_{a}(s) > \si^+_{a_\ga}(s)$ or
$\si^-_{a}(s) < \si^+_{a_\ga}(s)$.
\end{proof}

\smallskip

Notice that
\begin{equation}\label{siasia1}
  u \;\hbox{subsolution} \;\Longrightarrow \; \si^-(s) \leq u'(s) \leq
  \si^+(s) \;\;\hbox{for a.e. $s$.}
\end{equation}

We introduce four relevant functions:

\begin{eqnarray}
  s & \mapsto & \int_0^s \si^+_a(t) \, dt \label{siasia2} \\
  s & \mapsto & \int_0^s \si^-_a(t) \, dt \label{siasia3}\\
  s & \mapsto & - \int_s^1 \si^-_a(t) \, dt \label{siasia4}\\
  s & \mapsto & -\int_s^1 \si^+_a(t) \, dt \label{siasia5}.
\end{eqnarray}

\medskip
\begin{Remark}\label{four}According to \eqref{siasia1}, the function in \eqref{siasia2}
is the maximal (sub)solution to \eqref{HJg} vanishing at $s=0$, and
the one in \eqref{siasia3} the minimal (sub)solution  vanishing at
$s=0$. Analogously, the function defined in \eqref{siasia4} is  the
maximal (sub)solution vanishing at $s=1$, and  the one in
\eqref{siasia5} the minimal (sub)solution  vanishing at $s=1$.  All
of these functions
are of class $C^1$ because of Proposition \ref{siac1}.\\
We remark that when we write {\em maximal (sub)solution} et similia,
means that it is maximal in the class of subsolution to \eqref{HJg}
with a given property and it is, in addition, a solution to the
equation.
\end{Remark}
\smallskip

If $a=a_\ga$, it follows from \eqref{siasia} that  all of the above
functions coincide up to an additive
constant. We can state the following result.\\

\smallskip

\begin{Proposition}\label{equnique} The (sub)solution  to  \eqref{HJg}, with $a =
a_\ga$ is unique up to additive constants.\\
\end{Proposition}

\smallskip

From the properties of the solutions in \eqref{siasia2} and
\eqref{siasia3}, we directly derive
 a necessary
condition ({\it admissibility condition}) that two boundary data at
$0$
and $1$ must satisfy in order to  correspond to the values at the endpoints of  a subsolution to  \eqref{HJg}. \\

\begin{Lemma}  Assume that there  is a subsolution to \eqref{HJg} taking the values  $\al$ and $\be$ at $0$ and $1$,
then
\begin{equation}\label{compa}
   \int_0^1 \si_a^-(t)\,dt\leq  \be - \al \leq \int_0^1 \si_a^+(t)\,dt. \\
\end{equation}
\end{Lemma}

\medskip

The above condition is actually also sufficient:

\smallskip

\begin{Proposition}\label{statextra} Given   boundary data
$\al$,  $\be$,  satisfying \eqref{compa} the function $w$
\begin{equation}\label{compa00}
  s \longmapsto w(s):=\min \left \{ \al +
\int_0^s \si_a^+(t)\,dt , \, \be - \int_s^1 \si_a^-(t)\,dt \right \}
\end{equation}
is the unique solution to \eqref{HJg} taking the values $\al$ at
$s=0$, and  $\be$ at
$s=1$.
\end{Proposition}

The proof is in the Appendix  \ref{proofs}.\\

\subsection{Maximal subsolutions}\label{due}

The main result of this section is:

\begin{Proposition}\label{state}
 Assume that $w$ is a solution in $(0,1)$ to \eqref{HJg} for $a
 \geq a_\ga$,
continuously extended up to the boundary.  If
\begin{equation}\label{state0}
    H(1, \varphi'(1)) \geq a \qquad\hbox{for any $C^1$ supertangent $\varphi$ to $w$ constrained to $[0,1]$},
\end{equation}
then $w$ is the maximal (sub)solution taking the value $w(0)$ at
$0$. Namely:
\begin{equation}\label{state1}
   w(s)= w(0) + \int_0^s \si_a^+(t) \, dt \qquad \hbox{for $s \in
[0,1]$.}
\end{equation}
Conversely, if a solution $w$ is of the form \eqref{state1}, then
condition \eqref{state0} holds true.\\
\end{Proposition}

The proof is in the Appendix \ref{proofs}.

\smallskip

We fix $s_0 \in (0,1)$, by slightly generalizing the formulae
provided in the previous result and  arguing separately in the two
subintervals $[0,s_0]$ and $[s_0,1]$, we get:

\begin{Corollary}\label{fourbis} Let $s_0 \in (0,1)$. For any $\al \in \R$, the function
\[s \longmapsto  \left \{\begin{array}{cc}
                 \al - \int_s^{s_0} \si_a^-(t)\,dt & \quad\hbox{for $s \leq s_0$} \\
                 & \\
                  \al + \int_{s_0}^s \si_a^+(t)\,dt & \quad\hbox{for $s > s_0$} \\
                \end{array} \right .\]
 is the maximal subsolution to  \eqref{HJg} taking the value $\al$
 at $s_0$. It is in addition solution in $(0,1) \setminus \{s_0\}$,
 but the solution property fails at $s_0$, unless  $a = a_\ga$.

\end{Corollary}

\medskip

\begin{Remark}\label{five} In the light of Proposition
\ref{inverti} and  Remark \ref{four}, it is apparent that the
maximal solution to $H^\checkmark =a$  vanishing at $s=0$ is given
by
\[s \mapsto -\int_{1-s}^1 \si_a^-(t) \,dt.\]
This function satisfies the state constraint boundary condition at
$s =1$.
\end{Remark}

\bigskip

\subsection{Closed arcs}\label{tre}
In this subsection we assume that $\ga$ is a closed curve. Keeping
in mind Remark \ref{cgabis}, we aim at showing the existence of
periodic (sub)solution for any $a$  or, in other terms, that
periodic boundary conditions at $s=0$ and $s=1$ are admissible in
the sense of \eqref{compa}

Recall that $a \geq a_0 \geq c_\ga$.  We derive further information
in the case where {$a=a_0=c_\ga$. We will exploit the existence of
periodic subsolutions at the
 level $c_\ga$ in $(0,1)$, say, to fix ideas, vanishing at $0$
and $1$, as  pointed out in Remark \ref{cga}. These periodic
subsolutions are {\it sandwiched} in   between the function in
\eqref{siasia2} and the one in \eqref{siasia3}, according to Remark
\ref{four}. We derive:
\smallskip

\begin{Lemma}\label{six} We have
\begin{equation}\label{aubryall1}
   \int_0^1 \si_a^- (t) \,dt \leq 0 \leq \int_0^1 \si_a^+ (t) \,dt,
\end{equation}
and both the inequalities are strict if $a > c_\ga$.
\end{Lemma}

This in turn implies in view of  \eqref{compa}

\begin{Corollary} There are periodic
 solutions to \eqref{HJg} in $(0,1)$.
\end{Corollary}

\smallskip
Moreover:

\begin{Proposition}\label{aubryall}
\[ \min \left \{- \int_0^1 \si_{c_\ga}^- (t) \,dt,  \;\int_0^1 \si_{c_\ga}^+ (t)
\,dt \right \} =0.\]
\end{Proposition}
The proof is in the Appendix \ref{proofs}.

\medskip
From the previous result plus Proposition \ref{state} and Remark
\ref{five}, we derive the following.

\begin{Corollary}\label{seven} Let $a =c_\ga$  and $\al \in \R$;  then,  either the maximal solution
to $H=a$ taking the value $\al$ at $s=0$ or  the maximal solution to
$H^\checkmark =a$ taking the value $\al$ at $s=0$ are periodic.
\end{Corollary}

\medskip
In the final result  of the section we provide  a characterization
for the maximal periodic subsolution taking a given value at $s_0
\in (0,1)$. This corresponds, in the case of closed arcs,  to
Corollary \ref{fourbis}.

\begin{Corollary}\label{fourtris} Let $s_0 \in (0,1)$, $\al \in \R$,
we set
\[\be= \min \left \{- \int_0^{s_0} \si^-_a(t) \,dt,\;  \int_{s_0}^1 \si^+_a(t)
\,dt \right \}.\]
\begin{itemize}
    \item [{i)}] The maximal periodic subsolution to  \eqref{HJg} taking the value
$\al$ at $s_0$, denoted by $u$,  is uniquely determined by the
condition of being solution of the equation in $(0,s_0)$ and
$(s_0,1)$ taking the values $\al$ at $s_0$ and $\al+\be$ at $0$ and
$1$.
\item [{ii)}] If $\be = - \int_0^{s_0} \si^-_a(t) \,dt$, then
\begin{equation}\label{fourtris1}
    u(s) = \al - \int_s^{s_0} \si_a^-(t)\,dt  \quad\hbox{for $s \in
[0,s_0]$.}
\end{equation}
If instead $\be = \int_{s_0}^1 \si^+a(t) \,dt$, then
\begin{equation}\label{fourtris2}
   u(s) = \al + \int_{s_0}^s \si_a^+(t)\,dt  \quad\hbox{for $s \in
[s_0,1]$}
\end{equation}
\end{itemize}
\end{Corollary}
The proof  is  in the Appendix \ref{proofs}.

\bigskip

\subsection{From local to global}

The subsequent step in our analysis will be to transfer the
Hamilton--Jacobi equation from $\Gamma$ to the underlying graph
$\XX$, where it will take the form of a discrete functional
equation. In doing this, the relevant information we derive from the
above study is  the value at $s=1$ of the maximal solution to $H=a$
vanishing at $s=0$. It is given, in accordance with Proposition
\ref{state}, by
\[\int_0^1 \si_a^+(t)\,dt.\]
Therfore, if $\ga=\Psi(e)$ and $a \geq  a_\ga$, we define
\begin{equation}\label{sigra1}
   \si_a(e) := \int_0^1 \si_a^+(t)\,dt.
\end{equation}
(recall that $a \geq a_0 \geq c_\ga$).
\smallskip

Accordingly, we have
\begin{equation}\label{sigra2}
   \si_a(- e):= -\int_0^1 \si_a^-(t)\,dt.\\
\end{equation}

If $e$ is a loop, or equivalently $\ga=\Psi(e)$ a closed curve, we
summarize the information gathered in Propositions \ref{six} and
 \ref{aubryall} as follows:

\begin{Proposition}\label{loop} If $e$ is a loop then $\si_a(e) > 0$ for
$a > c_\ga$ and $\min  \big \{\si_{c_\ga}(e), \;\si_{c_\ga}(-e)\big
\} =0 $.
\end{Proposition}

Moreover, we  directly  deduce from definition and \eqref{siatris}
that

\begin{Lemma}\label{prestable} The function
\[ a \longmapsto \si_a(e)\]
is continuous and strictly increasing  in $[a_\ga, + \infty)$.\\
\end{Lemma}

\medskip

%%%%%%%%%%

\bigskip
\bigskip

\section{Global Part: the Discrete Functional Equation on the Abstract Graph}\label{statgraph}

In this section we push our analysis beyond the local existence of
solutions to \eqref{HJg} on each arc $\ga$, and study   the {
global} existence
of solutions to \eqref{HJ} on the whole network $\G$.\\

Let us start by noticing that if we consider $\VV$, the set of
vertices of $\Gamma$, it is easy to check that any solution $w$ to
\eqref{HJ} has a well defined {\it trace} $u= w_{|\VV}$ on $\VV$,
simply because of the continuity assumption. The following
uniqueness result is straightforward. We provide a proof in Appendix
A for reader'convenience.

\smallskip

\begin{Proposition}\label{unico} Let $u$ be a function defined on $\VV$, then
there exists at most one solution to  \eqref{HJ}  on $\G$
agreeing with $u$ on $\VV$.
\end{Proposition}
\smallskip

A converse property is by far more interesting, namely to find
conditions on a function defined on $\VV$ in order to (uniquely)
extend it on the whole network as solution to \eqref{HJ}.

This issue -- which is profoundly related to the {global} structure of the network -- will be carefully addressed in this section.\\

More precisely,  we study the problem of the admissibility, with
respect to the equation \eqref{HJ}, of a trace  $g: \VV
\longrightarrow \R$ defined on the { global} network and
characterize
all traces $g$ that can be continuously extended to solutions to
\eqref{HJ} on the whole $\Gamma$  as
solutions to  an appropriate {\it discrete functional equation} on
the  underlying abstract graph $\XX=(\VV, \EE)$.\\

\subsection{The discrete functional equation}

 Given $a \geq a_0$, the cochain $\si_a \in C^1(\XX,
\R)$ is defined as in \eqref{sigra1} where $e = \Psi^{-1}(\ga)$ and
$\Psi$ has been defined in \eqref{defPsi}.

\medskip

If we recall the admissibility condition introduced in \eqref{compa}
plus \eqref{sigra1}, \eqref{sigra2}, it is clear that the trace on
$\VV$ of a function $g:\Gamma \to \R$ admissible for the equations
on any arc satisfies
\begin{equation}\label{compabis}
    -\si_a(-e) \leq \d g(e) =g(\tt(e)) -  g(\oo(e)) \leq  \si_a(e)
\qquad\hbox{for any $e \in \EE$,}
\end{equation}
which in particular implies
$$
g(x)\leq \min_{e \in {\EE}_x} \big( g(\tt(e)) + \si_a(-e)\big)
\qquad\hbox{for
$x \in \VV$},\\
$$
where ${\EE}_x$ denotes the star centered at $x$, as defined in \eqref{defstar}.\\
\bigskip

Inspired by this, we introduce the following
{\it Discrete Functional Equation}:
\begin{equation}\label{HJa} \tag{$\mathcal{D}\mathit{FEa}$}
    u(x) = \min_{e \in {\EE}_x } \big(u(\tt(e)) + \si_a(-e)\big) \qquad\hbox{for $x
\in \VV$.}\\
\end{equation}

Observe that the formulation of the discrete problem takes somehow
into account the backward
character of viscosity solutions. \\

A function $v$ is solution  to \eqref{HJa} in some  subset $\VV'$ of
$\VV$ if \eqref{HJa}  holds true with $v$ in place of $u$ and $x
\in \VV'$.\\

A function $u:\VV \longrightarrow \R$ is a {\it subsolution} to
\eqref{HJa} if
 \begin{equation}\label{subsol}
u(x) \leq \min_{e \in {\EE}_x} \big(u(\tt(e)) + \si_a(-e)\big)
\qquad\hbox{for $x \in \VV$}
\end{equation}
  or,  equivalently, if for each $e\in \EE$ we have
\begin{equation}\label{subsolbis}
du(e) \leq \si_a(e)
\end{equation}
which is equivalent to ask that $u(\tt(e)) \leq  u(\oo(e)) + \si_a(e)$ for each $e\in\EE$.\\

\medskip
A subsolution is qualified as {\em strict}, if a strict inequality
prevails in \eqref{subsol}. \\

It is apparent that the property of being a solution or a
subsolution is not affected by addition of additive constants.

\smallskip

The crucial result linking the functional equation \eqref{HJa} to
 \eqref{HJ} is:

\smallskip
\begin{Proposition}\label{relationHJDEF}
Given $ a \geq a_0$,
\begin{itemize}
    \item [{i)}] any  solution to \eqref{HJa} in $\VV$ can be
(uniquely)  extended to a solution of \eqref{HJ} in $\Gamma$,
conversely the trace on $\VV$ of any solution of \eqref{HJ} in
$\Gamma$ is  solution to \eqref{HJa}:
    \item [{ii)}] any  subsolution to \eqref{HJa} in $\VV$ can be
  extended to a subsolution of \eqref{HJ} in $\Gamma$,
conversely the trace on $\VV$ of any sub solution of \eqref{HJ} in
$\Gamma$ is subsolution to \eqref{HJa}.
\end{itemize}
\end{Proposition}

\begin{proof} Assume that $u$ solves \eqref{HJa}. Let $x$, $y$ be
two adjacent vertices,   $e$ an edge with initial vertex $x$ and
final vertex $y$. We set $\ga = \Psi(e)$ and consequently
$\widetilde \ga = \Psi(-e)$, then $\ga(0)= \widetilde \ga(1) = x$
and $\ga(1)= \widetilde \ga(0)= y$. By the very definition of
(sub)solution to \eqref{HJa}, we have
\begin{eqnarray*}
  u(\ga(1))- u(\ga(0)) &\leq& \si_a(e) \\
 u(\ga(1))- u(\ga(0)) &= &u(\widetilde \ga(0)) - u(\widetilde \ga(1)) \geq  - \si_a(-e).
\end{eqnarray*}
Taking into account \eqref{compa}, we derive that the values
$u(\ga(0))$, $u(\ga(1))$ are admissible for \eqref{HJg} in $(0,1)$.
We therefore deduce from Proposition \ref{statextra} that there is
an unique solution, say $w: [0,1] \to \R$, to \eqref{HJg} taking
precisely these values at the boundary. We define
\[ v(z) = w(\ga^{-1}(z)) \qquad\hbox{for $z \in \ga((0,1))$.}\]
Since $\ga((0,1))= \widetilde \ga((0,1))$,  one needs to check  that
this definition is well-posed, performing the same construction for
$\widetilde \gamma$, but this is a direct consequence of Proposition
\ref{inverti}.

So far, we have successfully checked conditions {i)}, {ii)} in the
definition of solution to \eqref{HJ} (see Definition
\ref{defsolsubsol}). It is left to show { iii)}. Since $u$ is a
solution to \eqref{HJa}, for any  $x \in \VV$ there is an edge $e_0$
with $x$ as terminal vertex such that
\[ u(x) - u(\oo(e_0))= \si_a(e_0).\]
Taking into account \eqref{sigra1} and Proposition \ref{state},  we
deduce that, for $\ga=\Psi(e_0)$,  $v \circ \ga$ actually satisfies
the state constraint boundary condition in {iii)} with respect to
\eqref{HJg}.

 Conversely, let $u$ be a real function on $\VV$ which
is the trace on $\G$ of a solution to \eqref{HJ}.  It follows from
 the compatibility condition \eqref{compa}, and
the notations (\ref{sigra1})-(\ref{sigra2}), that $u$ is a
subsolution to \eqref{HJa}, {\it i.e.},
\begin{equation}\label{disequazione}
u(x)\leq \min_{e \in {\EE}_x} \big( u(\tt(e)) + \si_a(-e)\big)
\qquad\hbox{for $x \in \VV$}.
\end{equation}
In order to show that it is a solution to \eqref{HJa}, we need to
prove  that equality holds in \eqref{disequazione} for every
$x\in\VV$. In fact, since $u$ is the trace of a solution to
\eqref{HJ}, then it follows from condition {iii)} in Definition
\ref{defsolsubsol}, that for every vertex $x$ there is at least one
arc $\ga$ having $x$ as terminal point, such that   $ u(\ga(s))$
    satisfies the state constraint boundary   condition for \eqref{HJg}  at $s=1$. In particular, in the light of  Proposition \ref{state}, see \eqref{sigra1},
    this implies that there exists $e$
     with $\tt(e)=x$, or in other terms $-e\in\EE_x$, such that
     \[  u(x) - u(\oo(e)) = \si_a(e)\]
    or equivalently
    $$
    u(x) = u(\tt(-e)) + \si_a(e).
    $$
    Hence, equality holds in \eqref{disequazione},  and this completes the proof of item
    i). Item ii) can be proven arguing along the same lines.
\end{proof}

\medskip

The same argument as in the above proof allows also showing the
following:
\begin{Proposition} \label{prop5.8}
 Given $ a \geq a_0$ and $\VV' \subset \VV$, a function $u: \VV
\longrightarrow \R$ which is subsolution to \eqref{HJa} in $\VV$ and
solution in $\VV \setminus \VV'$  can be (uniquely)  extended to  a
function $v:\G \to \R$ subsolution of \eqref{HJ} in $\G$ and
solution in $\Gamma \setminus \VV'$. Conversely, the trace on $\VV$
of  a function $v:\G \to \R$, which is subsolution to \eqref{HJ} in $\G$ and
solution  in $\Gamma \setminus \VV'$, is a subsolution to \eqref{HJa}
in $\VV$ and a solution in $\VV \setminus \VV'$.
\end{Proposition}

\medskip

\subsection{Existence of solutions to ($\mathcal{D}FEa$) and critical value}

We want to introduce a notion of {\it critical value} for ($\mathcal{D}FEa$) and prove the existence of solutions.

Let us start by proving the following stability properties of solutions and subsolutions.

\smallskip

\begin{Proposition} \label{stable}\hfill
\begin{itemize}
    \item[{i)}] Let $a_n$ be a sequence in $\R$ converging to some $a$.  Let $u_n$ be subsolution
     to ($\mathcal{D}FEa_n$) for every $n$, with $u_n(x_0)$ bounded for some $x_0 \in \VV$; then
    $u_n$ converge, up to subsequences, to a subsolution to
    \eqref{HJa}.
    \item[{ii)}] Let $v_n$ be a sequence of solution to \eqref{HJa}, for
    some $a\in\R$, with $v_n(x_0)$ bounded for some $x_0 \in \VV$; then $v_n$ converges, up to a subsequence, to a
    solution to \eqref{HJa}.
\end{itemize}
\end{Proposition}

\medskip

\begin{proof}  Owing to the definition of subsolution and Lemma \ref{prestable}, we see that
\[  \langle \d u_n, e\rangle \leq  \si_b(e) \qquad\hbox{for every
$e \in \EE$,}\] where $b =\sup a_n$. This implies that the $\d
u_n$'s  are equibounded. We therefore get, exploiting the
boundedness assumption on $x_0$ and  Proposition \ref{equi} { ii)},
that $u_n$ is convergent, up to subsequences, to some $u$. In force
of Lemma \ref{prestable} we have
\[ u(\tt(e))- u(\oo(e)) - \si_a(e) = \lim_n  \big ( u_n(\tt(e))- u_n(\oo(e))  - \si_{a_n}(e) \big ) \leq
0\] for any $e$, showing that $u$ is subsolution to \eqref{HJa}.\\

Let now $v_n$ be a sequence of solutions to \eqref{HJa}; because of
the previous point,  $v_n$ converge, up to subsequences, to a
subsolution $v$ of the same equation. It is left to show that $v$ is
indeed a solution. Given $x \in \VV$, we find $e_n \in {\EE}_x$ with
\[ v_n (\tt(e_n)) -  v_n(x) - \si_a(- e_n) = 0.\]
Since the edges are finite, we deduce that there exists $e_0 \in \EE_x$  such that
\[e_n = e_0 \qquad\hbox{for infinitely many $n$.}\]
Up to extracting to a subsequence, passing to the limit as $n$ goes to infinity, we obtain
\[ v (\tt(e_0)) -  v(x) - \si_a(- e_0) = 0, \]
which completes the proof.
\end{proof}

\bigskip
\bigskip

We define the {\it critical value}  for \eqref{HJa}  (also called {\it Ma\~n\'e critical value}) as
\begin{equation}\label{defcritvalue}
c =c(\HN) :=  \min  \{a\geq a_0 \mid \hbox{\eqref{HJa} admits subsolutions}\}.
\end{equation}
First of all, notice that it is well-defined. In fact, because of the coercivity of the $H_\ga$'s,  $\si_a$ is strictly
positive for every $e$, when $a$ is large enough, so that any
constant function is a subsolution to \eqref{HJa}. This shows that
$c$ is finite. Note the minimum in the definition of $c$ is
justified by Proposition \ref{stable}, showing the existence of
critical subsolutions (namely, subsolutions to \eqref{HJa} with $a =
c$).\\

\medskip

The relevance of the critical value is apparent from the following result.

\begin{Proposition} \label{propcritvalue}
If there exists a  solution to  \eqref{HJa}, then $a=c$.
\end{Proposition}

\begin{proof} Clearly $a\geq c$, since every solution is also a subsolution. If $a > c$, then there exists  a strict subsolution $u$  to \eqref{HJa}. Let us assume, by contradiction, that
there exists also a solution $v$. Let $x_0$ be  point at which  $u -v$ achieves its maximum;
then
\begin{equation}\label{compaeik1}
    v(x_0) - v(\tt(e))  \leq u(x_0) - u(\tt(e))
\qquad\hbox{for any $e \in {\EE}_{x_0}$}.
\end{equation}
 By the very definition  of solution applied
to $v$, there is $e_0 \in {\EE}_{x_0}$ such that
\[    v(x_0) = v(\tt(e_0)) + \si_a(-e_0).\]
We derive, taking into account \eqref{compaeik1},
\[  u(x_0) \geq u(\tt(e_0)) + \si_a(-e_0),\]
which is in contrast with the very definition of strict subsolution.
\end{proof}

\smallskip
We further deduce a uniqueness result in the supercritical case.

\smallskip

\begin{Proposition} \label{prop5.11}
Let $ a > c$, $\VV' \subset \VV$. For any given \ function $u$
defined on $\VV'$ there is at most one solution $v$ of
\eqref{HJa} in $\VV\setminus\VV'$ agreeing with $u$ on $\VV'$.\\
\end{Proposition}

\begin{proof} Assume by contradiction that there are two distinct solutions $u_1$, $u_2$ both satisfying the
statement.
Being $a > c$, we know that there is a strict subsolution $w$ to
\eqref{HJa}. Therefore, given $\la \in (0,1)$ we have
\begin{equation}\label{nuit1}
  \la \, w(x) + (1- \la) \, u_1(x) < \min_{e \in \EE_x} \big( \la \, w(\tt(e)) + (1- \la) \, u_1(\tt(e)) \, +\, \si_a(-e)\big)
\end{equation}
for any $x \in \VV \setminus \VV'$. Up to interchanging the roles of
$u_1$ and $u_2$, we can assume that $\max_{\VV} (u_1 - u_2)
>0$, so that any maximizer is outside $\VV'$. For $\la$
sufficiently close to $0$, we still have that $ \big [ \la \, w +
(1- \la) \, u_1 \big ]-u_2$ achieves its maximum in $\VV \setminus
\VV'$. Let $x_0$ be one of these points of maximum; then, for every
$e\in {\EE}_{x_0}$ we have
\[\big [ \la \, w(x_0) +
(1- \la) \, u_1(x_0) \big ]-u_2 (x_0) \geq \big [ \la \, w(\tt(e)) +
(1- \la) \, u_1(\tt(e)) \big ]-u_2(\tt(e))\]
or
\[u_2(x_0) \leq u_2(\tt(e)) + \la \, w(x_0) + (1- \la) \, u_1(x_0) -\la \, w(\tt(e)) - (1- \la) \,
u_1(\tt(e)).\] Using  \eqref{nuit1} we can deduce
\[u_2(x_0) < \min_{e \in {\EE}_{x_0}}\big(  u_2(\tt(e))  -\si_a(-e)\big)\]
in contrast with $x_0 \not\in \VV'$ and $u_2$ being solution to
\eqref{HJa} in $\VV \setminus \VV'$.

\end{proof}

\bigskip

Given $a \geq a_0$, we define  for any path
$\xi=(e_1,\ldots,e_M)=(e_i)_{i=1}^M$
\begin{equation}\label{intrinsiclength}
\si_a(\xi)= \sum_{i=1}^M \si_a(e_i),
\end{equation}
and
\begin{equation}\label{Saa}
S_a(x,y):= \inf \{\si_a(\xi) \mid \xi\;\hbox{is a path linking $x$ to $y$}\}.
\end{equation}
The following triangle inequality is a direct consequence of the
 definition
\begin{equation}\label{triangle}
  S_a(x,y) \leq S_a(x,z)+ S_a(z,y) \qquad\hbox{for any $x$, $y$, $z$
  in $\VV$.}\\
\end{equation}

\bigskip

The next result starts unveiling the major role of cycles in the
forthcoming analysis.

\smallskip

\begin{Lemma}\label{lemmotto}  $S_a \not \equiv - \infty$ if and only if
\[\si_a(\xi) \geq 0 \qquad\hbox{ for any cycle $\xi$},\]
which is equivalent to say that $S_a(x,x)\geq 0$ for any $x\in \VV$.\\
\end{Lemma}

\begin{proof} If  $\si_a(\xi) < 0$ for some  cycle $\xi$,  then going
through it several times, we deduce that $S_a \equiv - \infty$.
Conversely, if  $\si_a(\xi) \geq 0$ for any cycle $\xi$, then
\[S_a(x,x)\geq0 \qquad   \hbox{for any $x \in \VV$}\]
and therefore $S_a \not\equiv - \infty$.
\end{proof}

\medskip

From the very definition of subsolution we derive the following result.

\smallskip

\begin{Proposition}\label{sa}  A function $u$ is a subsolution to \eqref{HJa} if
and only if
\[u(x)- u(y) \leq S_a(y,x) \qquad\hbox{for any $x,y \in \VV$.}\]
\end{Proposition}
\begin{proof}
 It follows easily from the definitions of subsolution in
\eqref{subsolbis} and $\si_a$ in  \eqref{intrinsiclength} that
\[u(x)- u(y)   \leq \si_a(\xi) \qquad\hbox{for any path $\xi$ linking
$y$ to $x$.}\] Taking the minimum over all  such  paths, we get the
inequality in the statement.  The converse is trivial, observing
that
$$S_a(\oo(e), \tt(e)) \leq \si_a(e) \qquad \hbox{for every $e\in \EE$}.$$
\end{proof}

\medskip

The previous result implies

\begin{Corollary}\label{corsuper} If $a \geq c$  then $S_a \not \equiv -
\infty$.\\
\end{Corollary}

\smallskip
Moreover:

\begin{Corollary}\label{coro} Given $a\geq c$,  $x$, $y$ in $\VV$, there exists a simple path $\eta$ with
$\oo(\eta)=x$, $\tt(\eta)=y$ such that $\si_a(\eta)= S_a(x,y)$.
\end{Corollary}
\begin{proof} Let $\xi=(e_i)_{i=1}^M$ be any path linking $x$ to $y$. If $\xi$ is
not simple there are indices $k > j$  such that $\tt(e_i) =
\tt(e_j)$.  We assume, to ease notations, that $k <M$, the case
$k=M$ can be treated with straightforward modifications.

We   have that $(e_i)_{i=j+1}^k$ is a cycle and the paths
$(e_i)_{i=1}^j$, $(e_i)_{i=k+1}^M$ are concatenated. We get,
according to Lemma \ref{lemmotto} that
\[\si_a(\xi) = \si_a \big ((e_i)_{i=1}^j \big ) + \si_a \big ((e_i)_{i=j+1}^k \big ) + \si_a
\big ((e_i)_{i=k+1}^M \big ) \geq \si_a \big ((e_i)_{i=1}^j \big )+
\si_a \big ((e_i)_{i=k+1}^M\big )\] and $(e_i)_{i=1}^j \cup
(e_i)_{i=k+1}^M$ is still a path linking $x$ to $y$. By iterating
the above procedure, we remove all cycles properly contained in
$\xi$ and  end up with a simple curve $\xi_0$ with $\oo(\xi_0)=x$,
$\tt(\xi_0)=y$ and $\si_a(\xi_0) \leq \si_a(\xi)$. This shows that
$S_a(x,y)$ can be realized as the infimum of simple paths from $x$
to $y$. Since there are finitely many of such paths, we get the
assertion.

\end{proof}

\bigskip

The condition in Corollary  \ref{corsuper} is actually necessary and
sufficient, as shown by the next result. In the proof we will use a
form of the basic {\it Bellman optimality principle} adapted to our
frame. It can be stated as follows: if $\xi= (e_i)_{i=1}^M$  is a
path   with
\[\si_a(\xi)= S_a(\oo(e), \tt(e))\]
 and $1 \leq j < k \leq M$, then $\eta := (e_i)_{i=j}^k$ satisfies
$\si_a(\eta)= S_a(\oo(e_j), \tt(e_k))$.

\smallskip

\begin{Proposition}\label{sasub}  Assume $S_a \not\equiv - \infty$. Given $y \in \VV$, the function
$u= S_a(y,\cdot)$ is solution to \eqref{HJa} in $\VV \setminus
\{y\}$ and subsolution  to \eqref{HJa} in $\VV$.\\
\end{Proposition}
\begin{proof}  The subsolution property comes from Proposition \ref{sa} and the triangle
inequality \eqref{triangle}. We proceed by showing that $u$ is solution
in $\VV \setminus \{y\}$.
 Let $x \neq y$,  then, by Corollary \ref{coro},  there is a  path $\xi= (e_i)_{i=1}^M$ linking $y$ to $x$
with
\[\si_a(\xi)= S_a(y,x).\]
By the Bellman optimality principle,  the path $\eta:= (e_i)_{i=1}^{M-1}$
satisfies
\[\si_a(\eta)= S_a(y, \tt (\eta)) = u(\tt (\eta)).\]
Consequently
\[u(x)= \si_a(\eta)+ \si_a(e_M)= u(\tt (\eta)) + \si_a(e_M)\]
with $- e_M \in {\EE}_x$. Hence
\[ u(x)- u(\tt(- e_M)) = u(x)- u(\tt (\eta)) = \si_a(e_M).\]
This concludes the proof.
\end{proof}
\smallskip

Using Proposition \ref{sa} and the triangle inequality
\eqref{triangle}, we also obtain

\begin{Corollary}\label{sasubbis} The function
\[ x \mapsto -S_c(x,y) \]
is a critical subsolution for any  fixed $y \in \VV$.
\end{Corollary}

\medskip

Combining Corollary \ref{corsuper} and Proposition \ref{sasub} we
get

\begin{Corollary} The distance  $S_a
\not\equiv - \infty$ if and only if $a \geq c$.\\
\end{Corollary}

\medskip
We further have

\medskip

\begin{Proposition}\label{cycle} Given $y \in \VV$, the function $x \longmapsto S_a(y,x)$ is solution to
\eqref{HJa} if and only if there exists a  cycle $\xi$ incident on
$y$ with $\si_a(\xi)=0$.
\end{Proposition}

\begin{proof}

($\Longrightarrow$) We will prove  in Proposition \ref{precritico} a more  general
property, namely that if the equation \eqref{HJa} admits a solution,
then there is a cycle $\xi$ with  with $\si_a(\xi)=0$.

\noindent ($\Longleftarrow$) Assume the existence of a  cycle, say
$\xi=(e_i)_{i=1}^M $, with $\si_a(\xi)=0$ incident on $y$.  Up to
relabelling the $e_i$'s, we can set $y= \oo (\xi) = \tt (\xi)$. We
claim that $u := S_a(y,\cdot)$ is a solution on the whole $\VV$. In
force of Proposition \ref{sasub}, it is enough to prove the
assertion at $y$. We have
\[0 \leq S_a(y,y)=u(y)   \leq \si_a(e_M) + S_a(y, \oo (e_M)) \leq \si_a (\xi),\]
and since $ \si_a (\xi) =0$, all the inequalities in the above
formula must indeed be equalities;  in particular
\[ u(y)- u(\tt(- e_M)) - \si_a(e_M) = u(y) - S_a(y, \oo (e_M))
- \si_a(e_M) = 0\] with $- e_M \in {\EE}_y$. This proves the claim.
\end{proof}

\bigskip
As announced, we complete the above proof by showing:

\begin{Proposition}\label{precritico} If the equation \eqref{HJa} admits a solution,  then there is a
 cycle $\xi$ with $\si_a(\xi)=0$.
\end{Proposition}
\begin{proof} Let us assume that $v$ is  a solution to \eqref{HJa}. Take
any $x \in \VV$; by the  definition of solution, we can find an edge
$e$
 with terminal vertex $x$  such that
\[v(x)- v( \oo (e)) = \si_a( e).\]
By iterating backward the procedure, we can construct  for any $M$ a
path $\xi = (e_i)_{i=1}^M$  such that
\begin{equation}\label{cy1}
   v( \tt (e_j))- v (\oo (e_k)) = \si_a \left ( (e_i)_{i=k}^j \right
) \quad\hbox{ for any  $j \geq k$.}
\end{equation}
 Since the graph is finite,
taking $M$ large enough, we have that for suitable indices $j > k$,
the path $(e_i)_{i=k}^j $ is a cycle, so that $v( \tt (e_j))- v (\oo
(e_k)) = 0$, and  the  relation \eqref{cy1} provides the assertion.
\end{proof}

\bigskip
The argument of the next proof is reminiscent of the one used for
the existence of critical solutions  of  Hamilton--Jacobi equations
in compact manifolds, see  \cite{FathiSiconolfi}.

\smallskip

\begin{Theorem}\label{critico} The critical equation ($\mathcal{D}\mathit{FEc}$) admits solutions.
\end{Theorem}

\begin{proof} We break the argument according to whether $c = a_0$ or $c > a_0$. Let us first discuss the first
instance. If in addition $c = a_\ga$ for some arc $\ga$, and we set
$e= \Psi^{-1}(\ga)$, then we get from  \eqref{siasia},
\eqref{sigra1}, \eqref{sigra2} that
\[\si_c(e \cup (-e))=0 .\]
  If instead $a_0= c_\ga$
for some closed arc $\ga$ of the network, then  $e = \Psi^{-1}(\ga)$
is a loop  and  we obtain  by   Remark \ref{loop}
\[\si_c(e)=0 \;\;\hbox{or}\;\; \si_c(-e)=0.\]
In both cases,  we infer the existence of a critical solution in the
light of  Proposition \ref{cycle}.

 We
proceed considering the case  $c > a_0$. Let us assume by
contradiction that there are no critical solutions. For any $y \in
\VV$, setting $u_y=S_c(y,\cdot)$, we can therefore find by
Proposition \ref{sasub} a positive constant $\de_y$ with
\begin{equation}\label{critico0}
   \max_{e \in {\EE}_y} \big( u_y(y) - u_y(\tt(e)) - \si_c( - e) \big ) = -
\de_y.
\end{equation}
We define $u = \sum_y \la_y \, u_y$, where
the $\la_y$ are positive coefficients summing to $1$, and set
\[\de = \min_y  \la_y \, \de_y. \]
 Exploiting that all the $u_y$'s are subsolutions on the
whole $\VV$ and using \eqref{critico0}, we conclude that for any $e
\in \EE$
\begin{eqnarray}
  && u(\tt(e)) - u(\oo(e)) - \si_c(e)  \nonumber \\ && =\; \sum_{y \neq \tt(e)} \la_y \big (
 u_y(\tt(e)) - u_y(\oo(e))  - \si_c(e) \big ) + \la_{\tt(e)} \big (
u_{\tt(e)}(\tt(e)) -u_{\tt(e)}(\oo(e))  - \si_c(e) \big )  \label{critico1}\\
     && \leq \; - \la_{\tt(e)} \, \de_{\tt(e)} \leq  -\de. \nonumber
\end{eqnarray}
Owing to Lemma \ref{prestable} and the fact that $c >a_0$, there is
$a_0 < b < c$ with
\[ \si_b(e) > \si_c(e) - \de \qquad\hbox{for every $e \in \EE$};\]
then we  deduce from \eqref{critico1} that
\[  u(\tt(e)) - u(\oo(e))  - \si_b(e) \leq 0 \qquad\hbox{for every $e$.}\]
This proves that $u$ is a subsolution to \eqref{HJa} with $a =b$,
which is impossible because $b < c$. Therefore the maximum in
\eqref{critico0} must be $0$ for some $y_0$, which in turn implies
that $S_c(y_0, \cdot)$ is a critical solution, as it was claimed.
\end{proof}

\smallskip

\begin{Remark}\label{cycy2}  Let $u$ be a  solution to $(\mathcal{D}\mathit{FEc})$. Let $e$ be a loop with $\oo(e)= \tt(e)=x$,
and $\ga=\Psi(e)$ is hence a closed curve. If  $c  < c_\ga$,
then, according to Remark \ref{loop}
\[0 = u(\oo(e)) -u(\tt(e)) < \si_c(e), \;\; 0 = u(\oo(-e)) -u(\tt(-e)) <
\si_c(-e)\] which shows that neither $e$ nor $-e$ realizes
\[\min_{e \in {\EE}_x } \big(u(\tt(e)) + \si_a(-e)\big).\]
This in turn implies that the edge $e$, and consequently $-e$, can be
removed from the edges of $\XX$ without affecting the status of
solution for $u$ or any other critical solution.

Things are different if $c =c_\ga$ because in this case, see Remark
\ref{loop},
\[0= \min \{\si_c(e),\si_c(-e) \}= u(\oo(e)) - u(\tt(e))= u(\oo(-e))
-u(\tt(-e)).\]\\
\end{Remark}

\medskip

\subsection{The Aubry set and some structural properties of  solutions}
Inspired by what discussed in the previous subsection, we introduce
the following definition.

\begin{Definition}
The {\it Aubry set} is defined as
\begin{equation}\label{Aubry1}
\A_\XX^*= \A_\XX^*(\HN)= \{e \in \EE \mid \;\hbox{belonging to some
cycle with $\si_c(\xi) = 0$}\}.
\end{equation}
The {\it projected Aubry set} is given by
\begin{equation}\label{Aubry2}
\A_\XX =\A_\XX(\HN)= \{y \in \VV \mid \exists \,  \xi \;\hbox{ cycle
incident on $y$ with $\si_c(\xi) = 0$}\}.
\end{equation}
The projected Aubry set is partitioned in {\it static classes},
defined  as the equivalence classes with respect to the relation
\[S_c(x,y)+S_c(x,y)=0.\]
Equivalently $x$ and $y$ belong to the same static class if there is
a cycle $\xi$ with $\si_c(\xi)=0$ incident on both of them; in
particular, the whole cycle $\xi$ is then contained in this static
class.

\end{Definition}

\smallskip

\begin{Remark}\label{classau} Clearly, $x \in \A_\XX$ if and only if $x= \oo(e)=\tt(e')$, for
some $e$, $e'$  in  $\A_\XX^*$; moreover, if  $e \in \A_\XX^*$, then
$\oo(e)$ and $\tt(e)$ belong to $\A_\XX$. The converse of this last
property  is not true because, for instance, if $e \in \A_\XX^*$
then $-e$ might not belong to $\A_\XX^*$.  It is also
possible to have a pair of adjacent vertices belonging to different
static classes of $\A_\XX$ linked by an edge not in $\A_\XX^*$, or
even vertices of the same static classes linked by multiple edges
not all belonging
to $\A_\XX^*$. \\
\end{Remark}
\smallskip

We immediately derive  from Proposition \ref{precritico} and
Theorem \ref{critico} the following result.\\

\begin{Lemma}\label{Aubrynonempty} The Aubry sets are nonempty. Moreover
\[\A_\XX =\{y\in \VV:\; S_c(y,y)=0\}= \{y\in \VV:\; S_c(y,\cdot) \;\hbox{is solution to $(\mathcal{D}\mathit{FEc})$}\}.\]
\end{Lemma}

\medskip

\medskip

We have a structural  result on critical solutions. By admissible
trace $g$ on  $\VV' \subset \VV$ (for the critical equation), we
mean a function satisfying

\begin{equation}\label{compcomp}
   g(x)- g(y) \leq S_c(y,x) \qquad\hbox{for any $x$, \,$y$ in $\VV'$}.
\end{equation}

\medskip

\begin{Theorem}\label{uniquenessset}
 Given an admissible trace $g$ on $\A_\XX$, the unique
solution to $(\mathcal{D}\mathit{FEc})$ taking the value $g$ on
$\A_\XX$ is
\begin{equation}\label{formsol}
  v (x):= \min \{g(y) + S_c(y,x) \mid y \in \A_\XX\}.
\end{equation}
In particular, $\A_\XX$ represents a uniqueness set for the
equation.\\
\end{Theorem}

\begin{proof}  Taking into account \eqref{compcomp} and the fact that $S_c(y,y)=0$ for any $y \in \A_\XX$, we deduce that $g$ and $v$ coincide on $\A_\XX$.
The function $v$ is a  subsolution in force of Proposition \ref{sa}.
Take $x_0 \in \VV$, then
\[v(x_0) = g(y_0)+ S_c(y_0,x_0) \qquad\hbox{for some $y_0 \in \A_\XX$.}\]
We know that the  function
\[ \psi(x)= g(y_0) + S_c(y_0,x)\]
is a critical solution, in addition $x_0$ is a maximizer of $v
-\psi$ in $\VV$,  consequently
\[  \psi(x_0) - \psi(\tt(e))  \leq  v(x_0)- v(\tt(e))
\qquad\hbox{ for any $e \in {\EE}_{x_0}$.}\] Since $\psi$ is
critical solution, there is $e_0 \in {\EE}_{x_0}$ with
\[0=   \psi(x_0) - \psi(\tt(e_0)) - \si_c(- e_0) \leq v(x_0)- v(\tt(e_0)) - \si_c(-
e_0).\] Since $v$ is a subsolution, the inequality in the above
formula must be an equality. This shows that $v$ is a critical
solution.

Assume now that  $w$ is another solution agreeing with $g$ on
$\A_\XX$. Given any $x \in \VV$, we construct, arguing as in
Proposition \ref{precritico},  a path $\xi= (e_i)_{i=1}^M$ with $\tt
(\xi)=x$ and such that
\[
   w( \tt (e_j))- w (\oo (e_k)) = \si_c \left ( (e_i)_{i=k}^j  \right
) \quad\hbox{ for any  $j \geq k$.}\] If $M$ is sufficiently  large,
there must exist $ j_0 \geq k_0$ such that $ (e_i)_{i=k_0}^{j_0}$ is a
cycle. We deduce that there are $y \in \A_\XX$  and a path $\eta$
linking $y$ to $x$ with
\[w(x) = w(y) + \si_c(\eta) \geq g(y)+ S_c(y,x) \geq v(x).\]
Since the converse inequality holds true  by Proposition \ref{sa},
we get $w(x)= v(x)$. This ends the proof.
\end{proof}

\smallskip

We record for later use an immediate consequence of the above
result:
\begin{Corollary}\label{corunique}
 Given  $\VV' \subset \A_\XX$, and an admissible trace $g$ on it, the function
\begin{equation}\label{formsol}
  v (x):= \min \{g(y) + S_c(y,x) \mid y \in \VV'\}
\end{equation}
is a  solution to $(\mathcal{D}\mathit{FEc})$ taking the value $g$
on  $\VV'$.\\
\end{Corollary}

\medskip

We can also derive a representation formula for solutions at $a
> c$ in some subset of $\VV$. To help understanding the next
statement, we recall that $S_a(x,x) >0$ for any $x \in \VV$ whenever
$a > c$.

\medskip

\begin{Theorem} \label{prop5.23}
Let $a > c$, $\VV' \subset \VV$. Let $g$ be  a function defined on
$\VV'$ satisfying \eqref{compcomp} with $S_a$ in place of $S_c$,
then  the function
\[ v(x)= \left \{\begin{array}{cc}
    g(x) & \quad\hbox{if $x \in \VV'$} \\
   \min  \{g(y) + S_a(y,x) \mid y \in \VV'\} & \quad\hbox{if $x \not\in \VV'$} \\
    \end{array} \right .\]
is the unique solution to \eqref{HJa} in $\VV\setminus\VV'$ agreeing
with $g$ on $\VV'$.
It is in addition subsolution on the whole of  $\VV$.\\
\end{Theorem}
\begin{proof}
 We claim that
 \begin{equation}\label{5231}
    v(z) -v(x) \leq S_a(x,z) \qquad\hbox{for any $z$, $x$ in $\VV$.}
\end{equation}
The property is true by assumption if both $z$, $x$ are in $\VV'$,
if instead $z$, $y$ are in $\VV \setminus \VV'$ we have
\[v(z)- v(x) \leq g(y) + S_a(y,z) -  g(y) - S_a(y,x) \leq S_a(x,z),\]
where $y \in \VV'$ is optimal for $v(x)$ and we have exploited the
triangle inequality  \eqref{triangle}. If $z \not \in \VV'$, $x \in
\VV'$, then \eqref{5231} directly comes from the very definition of
$v$. Finally, if $z  \in \VV'$, $x \not\in \VV'$, we denote by $y$
an optimal element in $\VV'$ and use the triangle inequality to
write
\[v(z)- v(x) = g(z)  -  g(y) - S_a(y,x) \leq S_a(y,z) - S_a(y,x) \leq S_a(x,z). \]
This concludes the proof of claim \eqref{5231} and therefore shows,
according to Proposition \ref{sa}, that $v$ is a subsolution in
$\VV$. Taking into account that $S_a(y, \cdot)$ is solution in $\VV
\setminus \VV'$, we also get, arguing as in Theorem
\ref{uniquenessset}, that $v$ is solution in $\VV \setminus \VV'$.
Uniqueness follows from Proposition \ref{prop5.11}.
\end{proof}

\bigskip

\section{Back to the network}\label{backlocal}

In this section we switch our attention back to the network
$\Gamma$,  or in other terms, we give again visibility, besides the
vertices,  to the interior points of the arcs.  We combine the
global information gathered on the abstract graph with the outputs
of the local analysis on the arcs of the network. We define an
appropriate notion of Aubry set and provide a PDE characterization
of its points.

 Exploiting the  richer (differentiable) structure of
$\G$, we  establish, on the basis of our findings in the previous
section, some regularity properties for critical subsolutions and
solutions. This will generalize what is known  for the continuous
case in the framework of Weak KAM theory, see for example
\cite{Fathi}. Finally, we give specific uniqueness results and
representation formulae for solutions on the network.

\smallskip

\subsection{Subsolutions and solutions on $\G$} \; The next
result shows, as pointed out already in the Introduction,  how the
notion of solution to \eqref{HJ} can be recovered from the notion of
subsolution. The relevance of the issue is that the latter just requires the usual subsolution property on any arc and
continuity at the junctures.  The argument significantly
illustrates the interplay between the immersed network and underlying abstract
graph.

\smallskip

\begin{Theorem}\label{deniro} Let $a \geq c$ and  $y \in \G$, then the
maximal subsolution  to \eqref{HJ} attaining a given value at $y$ is
solution in $\G \setminus \{y\}$.
\end{Theorem}
\begin{proof} We can assume $y \in \G \setminus \VV$ otherwise
the assertion  is a consequence of Propositions \ref{sa},
\ref{sasub} and Proposition \ref{prop5.8} with $\VV'=\{y\}$. It is
not restrictive to take $0$ as value assigned at $y$.  We therefore
denote by $v$ the maximal subsolution vanishing at $y$, see
Proposition \ref{predeniro}. We select $\ga \in \EN$ such that $y=
\ga(s_0)$ for some $s_0 \in (0,1)$, and set $e=\Psi^{-1}(\ga)$. We
first assume that $\ga$ is not a closed arc. Since $v$ must be in
particular subsolution in the arc $\ga$, we have by Corollary
\ref{fourbis}
\begin{eqnarray*}
  v(\ga(1)) &\leq& \int_{s_0}^1 \si_a^+(t)\,dt=:\be \\
  v(\ga(0)) &\leq& - \int_0^{s_0} \si_a^-(t)\,dt=: \al,
\end{eqnarray*}
where $\si_a^+$, $\si_a^-$ are defined as in \eqref{sia},
\eqref{siabis}. The maximal admissible trace $g$, in the sense of
\eqref{compcomp}, on $\VV':= \{\oo(e), \tt(e)\}$ dominated by  $
\al$ at $\oo(e)=\ga(0)$, and $\be$ at $\tt(e)=\ga(1)$   is
\begin{eqnarray*}
  \al^* &:=& \min\{\al, \; \be+ S_a(\tt(e),\oo(e)) \}\\
  \be^* &:=& \min\{\al, \; \be+ S_a(\oo(e),\tt(e))\}.
\end{eqnarray*}
 According to Proposition \ref{sa}, Theorem \ref{prop5.23} and Corollary \ref{corunique}, the
function $w: \VV \to\R$ defined as
\[ w(x)= \left \{
\begin{array}{cc}
   \al^* & \quad\hbox{if $x =\oo(e)$} \\
   \be^* & \quad\hbox{if  $x =\tt(e)$} \\
   \min \{\al^* + S_a(\oo(e),x) , \, \be^* + S_a(\tt(e),x) \} & \quad\hbox{if $x \neq \oo(e)$ and $x \neq \tt(e)$} \\
    \end{array}
 \right .\]
is the maximal subsolution  to \eqref{HJa} on $\VV$ agreeing with
$\al^*$, $\be^*$ at the vertices of $e$.   It  is in addition
solution in $\VV \setminus \{\ga(0),\ga(1)\}$.  By Proposition
\ref{prop5.8} it can thus be extended to a subsolution of \eqref{HJ}
in $\G$, denoted by $\ov w$, which is in addition solution  in $\G
\setminus \{\ga(0),\ga(1)\}$. The function $\ov w$ is the maximal
subsolution to \eqref{HJ} taking the values $\al^*$, $\be^*$ on the
vertices of $\ga$,  but it does not necessarily  vanish at $y$. We
have in any case
\begin{equation}\label{mus0}
  v \leq \ov w \qquad\hbox{in $\G $.}
\end{equation}
To complete the proof, we need  to suitably adjust $\ov w$ inside
$\ga$ in order to attain the value $0$ at $y$. To this end,  we
proceed by showing that the boundary data $\al^*$, $0$ and $0$, $\be^*$
are admissible, in the sense of \eqref{compa}, for \eqref{HJg}
restricted to the subintervals $[0,s_0]$ and $[s_0,1]$,
respectively. In fact,
\begin{equation}\label{mus1}
   \al^* \leq  \al= - \int_0^{s_0} \si_a^-(t)\,dt,
\end{equation}
and if a strict inequality prevails in the above formula, we get
\begin{equation}\label{mus1bis}
   \al^*=   \int_{s_0}^1 \si_a^+(t)\,dt   + S_a(\tt(e),\oo(e)).
\end{equation}
Let  us consider a cycle in $\XX$ of the form $\xi \cup e$, where
$\xi$ be a path in  linking $\tt(e)$ to $\oo(e)$ with $\si_a(\xi)=
S_a(\tt(e),\oo(e))$,  see  Corollary \ref{coro}. Then $\si_a(\xi
\cup e) \geq 0$ and consequently $S_a(\tt(e), \oo(e)) \geq -
\si_a(e)$. By plugging this relation in \eqref{mus1bis} and
recalling the definition of $\si_a(e)$, we get
\begin{equation}\label{mus2}
\al^* \geq    \int_{s_0}^1 \si_a^+(t)\,dt - \int_0^1 \si_a^+(t)\,dt
= -\int_0^{s_0} \si_a^+(t)\,dt.
\end{equation}
By combining \eqref{mus1}, \eqref{mus2}  we  have
\[   \int_0^{s_0} \si_a^-(t)\,dt \leq  -\al^* \leq \int_0^{s_0}
\si_a^+(t)\,dt,\] proving the claimed admissibility property in
$[0,s_0]$. A straightforward modification of the previous argument
shows the same in $[s_0,1]$.  Thus, there exists  a function $u$ on
$\ga([0,1])$ uniquely determined by requiring $u \circ \ga$ to be
solution to \eqref{HJg} in $(0,s_0)$ and $(s_0,1)$, and in addition
to take the values $\al^*$, $0$, $\be^*$ at $\ga(0)$, $y$, $\ga(1)$,
respectively. This is also the maximal subsolution of \eqref{HJg} in
$(0,1)$  taking such values at the boundary points and at $s=s_0$.
The function
\[\ov{\ov w}(x)= \left \{
\begin{array}{cc}
   \ov w & \quad\hbox{in $\G \setminus \ga[0,1]$} \\
   u & \quad\hbox{in $\ga[0,1]$} \\
    \end{array}
 \right .\]
 is subsolution  to \eqref{HJ} in $\G$ and by the maximality property of
 $u$ on $\ga$ and \eqref{mus0}
\[ v \leq \ov{\ov w} \qquad\hbox{in $\G$,}\]
 which immediately implies $v=\ov{\ov w}$.

 The function $v$  is
 by construction solution to \eqref{HJ} in $\G \setminus
 \{\ga(0),y,\ga(1)\}$. Moreover, taking into account Remark \ref{four} and Proposition
\ref{state} applied to the subinterval $[0,s_0]$, we see that if
$\ov w(\ga(0))= \al$ then  $\ov w$ satisfies condition iii) in
definition of solution to  \eqref{HJ} at $\ga(0)$ with respect to
the arc $\widetilde \ga$. If instead $\ov w(\oo(e)) = \al +
S_a(\tt(e),\oo(e))$ then again condition iii) of definition of
solution is satisfied with respect to some arc different from $\ga$,
$\widetilde \ga$ because of Propositions \ref{sasub} and
\ref{prop5.8}. Similarly, we prove that $v$ is  solution at
$\ga(1)$. This concludes the proof if $\ga$ is not a closed arc.

If instead $\ga$ is a closed arc, then we indicate by $w$ the
maximal periodic subsolution of \eqref{HJg} in $(0,1)$ vanishing at
$s=s_0$, see Corollary \ref{fourtris}. Arguing as in the first part
of the proof, we see that the maximal subsolution $v$  to \eqref{HJ}
in $\G$ vanishing at $y$ is given by
\[ v(x)= \left \{
\begin{array}{cc}
   w(\ga^{-1}(x)) & \quad\hbox{in $\ga([0,1])$} \\
   w(\ga(0)) + S_a(\ga(0),x) & \quad\hbox{in  $\G \setminus \ga([0,1])$}. \\
    \end{array}
 \right .\]
 Taking into account the representation formulae for $w$ provided in
 item ii) of Corollary \ref{fourtris} and arguing again as in the
 first part of the proof, we show that $v$ is solution to \eqref{HJ}
 in $\G \setminus \{y\}$, as it was claimed.

\end{proof}

\medskip

\subsection{ Aubry set in $\G$}

We define the Aubry set $\A_\Gamma$ on the network as
\begin{equation}\label{Aubrynetwork}
\A_\Gamma:= \big \{ x \in \R^N \mid x= \Psi(e)(t) \;\hbox{for some
$e \in \A_\XX^*$, $t \in [0,1]$} \big \}.\\
\end{equation}

\medskip

One could also consider a lift of $\A_\Gamma$ to the tangent bundle
$T \Gamma$, as in continuous case. {For example, this could  be
useful to study the analogues in this setting of  Mather's
measures, Mather sets, minimal average actions, etc. (see for
example \cite{Fathi, Sorrentinobook} for precise definitions); this
discussion, however, would go beyond our current objectives, so we
decided to postpone it  to a future investigation.\\

\begin{Remark}\label{regolo}
We point out for later use that the support of an arc $\ga$ belongs
to $\A_\Gamma$ if and only if $\ga=\Psi(e)$ and
at least one between  $e$ or $-e$ is in $\A_\XX^*$.\\
\end{Remark}

\smallskip

The first lemma regards  subsolutions to the critical equation on
$\XX$. Briefly, it says that -- analogously to what happens in the
continuous case, see \cite{Fathi} --  the differential of a critical
subsolution is prescribed on the Aubry set and that  critical
subsolutions are never strict on the Aubry set. On the other hand,
it is always possible to find critical subsolutions that are strict
outside
 the Aubry set. This will be used in the next subsection to obtain the same results on networks. See
Theorems \ref{regolauno}, \ref{regolaregola}. \\

\begin{Lemma}\label{preregola} Given a subsolution $u$ to
$(\mathcal{D}\mathit{FEc})$, one has
\begin{equation}\label{preregola00}
    \langle \d u, e \rangle = \si_a(e)  \qquad\hbox{for any $e \in
    \A_\XX^*$.}
\end{equation}
Furthermore, there exists  a subsolution $w$ to
$(\mathcal{D}\mathit{FEc})$ with
\begin{equation}\label{preregola000}
   \langle \d w, e \rangle < \si_a(e)   \qquad\hbox{for any  $e \in \EE \setminus
 \A_\XX^*$.}
\end{equation}
\end{Lemma}
\begin{proof}
Let  $u$ be    a critical subsolution and assume for purposes of
contradiction that
\[  \langle \d u, \ov e \rangle < \si_a(\ov e) \qquad\hbox{for some $\ov e\in \A_\XX^*$.}\]
By  the very definition of Aubry set, we can find   a cycle
$\xi=(e_i)_{i=1}^M$ such that $\ov e=e_j$  for some $j=1,\cdots, M$
and $\si_c(\xi)=0$. Taking into account that $u$ is a subsolution,
we   have
\[ \langle \d u, e_i \rangle \leq \si_a(e_i)  \;\;\; \hbox{for $i \neq j$ \; and  } \;\;    \langle \d u, e_j \rangle <
\si_a(e_j).\] This implies
\[ 0 = \sum_i \langle \d u, e_i \rangle < \sum_i \si_c(e_i)=
\si_c(\xi)=0,\] which is impossible. We pass to the second part of
the statement. We start constructing for any $e_0  \in \EE \setminus
\A_\XX^*$ a critical subsolution $u_{e_0}$ with
\begin{equation}\label{preregola1}
   \langle\ d u_{e_0},e_0 \rangle < \si_a(e_0).
\end{equation}
The argument will be organized taking into account the
classification of edges in $\A^*_\XX$ provided in Remark
\ref{classau}. If $\tt(e_0) \not\in \A_\XX$, then we set $u_{e_0}=
S_c(\tt(e_0),\cdot)$, according to Lemma \ref{Aubrynonempty},
$u_{e_0}$ is not a critical solution at $\tt(e_0)$ which implies
\eqref{preregola1}. If $\tt(e_0) \in \A_\XX$, we consider the
critical subsolutions $S_c(\tt(e_0),\cdot)$ and $- S_c(\cdot,
\tt(e_0)$, see Proposition \ref{sasub} and Corollary \ref{sasubbis}.
Taking into account the characterization of $\A_\XX$ given in  Lemma
\ref{Aubrynonempty}, we have
\begin{eqnarray*}
-S_c(\tt(e_0),\oo(e_0)) =  S_c(\tt(e_0),\tt(e_0))- S_c(\tt(e_0),\oo(e_0)) &\leq& \si_c(e_0) \\
 S_c(\oo(e_0),\tt(e_0)) = -S_c(\tt(e_0),\tt(e_0))+ S_c(\oo(e_0),\tt(e_0)) &\leq&
 \si_c(e_0).
\end{eqnarray*}
If equality prevails in both above formulae, we get
\[ S_c(\oo(e_0),\tt(e_0))+ S_c(\tt(e_0),\oo(e_0))=0\]
which is possible if and only if  both $\oo(e_0)$, $\tt(e_0)$ are in
the Aubry set and belong to the same static class.  If this is not
the case,  we satisfy \eqref{preregola1} up to choosing $u_{e_0}$
equals to $S_c(\tt(e_0),\cdot)$ or $- S_c(\cdot, \tt(e_0))$. If
instead the two vertices are  in the same static class, we claim
that
\begin{equation}\label{preregola3}
  S_c(\tt(e_0),\tt(e_0)) - S_c(\tt(e_0),\oo(e_0)) =
- S_c(\tt(e_0),\oo(e_0)) < \si_c(e_0).
\end{equation}
In fact, we know, by the very definition of static class, that there
is a path $\xi$ linking $\tt(e_0)$ to $\oo(e_0)$ with all the edges
belonging to $\A_\XX^*$. Therefore,  using  Lemma
\ref{Aubrynonempty} and the first part of the statement that we have
just proven, applied to the critical subsolution
$-S_c(\cdot,\oo(e_0))$, we have that
\[  S_c(\tt(e_0), \oo(e_0)) = - S_c(\oo(e_0), \oo(e_0))+ S_c(\tt(e_0),
\oo(e_0))= \si_c(\xi).\]
Were \eqref{preregola3} false, we should further have
\[0 = - S_c(\tt(e_0), \oo(e_0)) + S_c(\tt(e_0), \oo(e_0))=
\si_c(\xi\cup e_0)\] and consequently $e_0\in \A_\XX^*$, which is
impossible. Formula \eqref{preregola1} is therefore satisfied with
$u_{e_0} = S_c(\tt(e_0),\cdot)$.  This completes the proof of
\eqref{preregola1}.

We conclude arguing along the same lines of Theorem \ref{critico}.
Given $e \in \EE \setminus \A_\XX^*$, we denote by $u_e$ a critical
subsolution satisfying \eqref{preregola1} with $e$ in place of
$e_0$.  We choose positive constants $\la_e$, for $e \in \EE
\setminus \A_\XX^*$, summing to $1$, and  define a critical
subsolution via
\[w = \sum_{e \in \EE \setminus \A_\XX^*} \la_e \,u_e.\]
Given $e_0\in \EE \setminus \A_\XX^*$, we have
\[\langle \d w, e_0 \rangle = \sum_{e \neq e_0} \la_e \, \langle \d
u_e,e_0 \rangle + \la_{e_0} \, \langle \d u_{e_0},e_0 \rangle <
\sigma_c(e_0),\] as we wished to prove.
\end{proof}

We derive a PDE characterization of points in the Aubry set,
generalizing a property of the continuous case.

\smallskip

\begin{Proposition}\label{postdeniro}  The maximal subsolution to $(\mathcal{H}J\mathit{c})$ taking a given value at a point $y \in \G$  is   a
critical solution on the whole network if and only if $y \in \A_\G$.
\end{Proposition}
\begin{proof} If $y \in \VV$, the assertion comes from Lemma \ref{Aubrynonempty}, we can then assume from now on that $y \in \G \setminus \VV$.
 We prescribe, without loss of generality,  the value $0$ at $y$, and  denote by $v$
the maximal subsolution vanishing at $y$, see  Proposition
\ref{predeniro}. We denote by $\ga$ an arc whose support contains
$y$.  \\
We first assume that $\ga$ is not a closed curve.
Taking into account Theorem  \ref{deniro}, it is enough to show that
$v$  is solution at $y$ if and only if $y \in \A_\G$. Looking at the
proof of Theorem  \ref{deniro}, we see that the solution property at
$y$ is in turn equivalent to the following:  the solution of
$(HJ_\ga \mathit{c})$
 in $(0,1)$ taking the values  $v(\ga(0))$, $v(\ga(1))$
 at $0$, $1$,  respectively, vanishes at $s=s_0$.
In the light of Proposition \ref{statextra}, this  boils down to
show
\begin{equation}\label{postdeniro1}
  \min  \{ v(\ga(0)) + {\mathbf A}, \, v(\ga(1)) - {\mathbf B}
 \} =0,
\end{equation}
 where $\si_c^+$, $\si^-_c$ are defined as in \eqref{sia},
\eqref{siabis}, respectively, and
\[  {\mathbf A}  = \int_0^{s_0} \si_c^+(t)\,dt \qquad\qquad {\mathbf B}  = \int_{s_0}^1 \si_c^-(t)\,dt.  \]
Taking into account the proof of Theorem \ref{deniro}, we know that
\begin{eqnarray}
  v(\ga(0)) &=&  \min \{ - {\mathbf D },\;  {\mathbf C} +S_c(\ga(1), \ga(0))  \} \label{postdeniro00}\\
  v(\ga(1)) &=& \min  \{  {\mathbf C} ,\; - {\mathbf D} + S_c(\ga(0), \ga(1)) \} \label{postdeniro0}
\end{eqnarray}
where
\[{\mathbf C}  =\int_{s_0}^1 \si_c^+(t)\,dt \qquad\qquad  {\mathbf D}  = \int_0^{s_0} \si_c^-(t)\,dt.  \]
Then
\begin{equation}\label{postdeniro2}
{\footnotesize
  v(\ga(0))  + {\mathbf A } =  \left \{\begin{array}{ll}
   \int_0^{s_0} [ \si_c^+(t) -\si_c^-(t)] \, dt &\quad\hbox{if $v(\ga(0))= - {\mathbf D }$}\\
   & \\
      \int_0^1 \si^+_c(t) \,dt  +S_c(\ga(1), \ga(0))  & \quad\hbox{if $v(\ga(0))= {\mathbf C} +S_c(\ga(1), \ga(0))$} \\
   \end{array} \right.}
\end{equation}
and
\begin{equation}\label{postdeniro3}
{\footnotesize
    v(\ga(1)) - {\mathbf B} =  \left \{\begin{array}{ll}
   \int_{s_0}^1 [ \si_c^+(t) -\si_c^-(t)] \,dt &\quad\hbox{if $v(\ga(1))=  {\mathbf C }$}\\
   & \\
    -\int_0^1 \si^-_c(t) \,dt  +S_c(\ga(0), \ga(1))  & \quad\hbox{if $v(\ga(1))= -{\mathbf D} +S_c(\ga(0), \ga(1))$}. \\
   \end{array} \right.}
\end{equation}
 Exploiting the property that $\si_c(\xi) \geq 0$ for
any cycle $\xi$ in $\XX$,   we see that
\begin{eqnarray*}
  S_c(\ga(0),\ga(1)) &\geq& -\si_c(-e)= \int_0^1 \si^-_c(t) \,dt  \\
  S_c(\ga(1),\ga(0)) &\geq& -\si_c(e)= -\int_0^1 \si^+_c(t) \,dt.
\end{eqnarray*}
Equality holds in the first formula if and only if there is a cycle
$\xi$ with $-e \subset \xi$,  $\si_c(\xi)=0$, and in the second one
if and only if if there a cycle $\eta$ with $e \subset \xi$,
$\si_c(\eta)=0$. We in addition  have that
\[ \int_0^{s_0} [ \si_c^+(t) -\si_c^-(t)] \, dt =0 \quad\hbox{or}\quad  \int_{s_0}^1 [ \si_c^+(t) -\si_c^-(t)] \, dt=0\]
if and only if $c=a_\ga$,  and   this case both $e$
and $-e$ belong to $\A^*_\XX$.  In the light of the above remarks,
\eqref{postdeniro2},  \eqref{postdeniro3}, we conclude
that \eqref{postdeniro1} holds if and only if $y \in \A_\G$.

 This concludes the proof when $\ga$ is
not a closed arc. The argument for $\ga$ closed arc goes along the
same lines just adapting the representation formulae for solutions
of $(HJ_\ga \mathit{c})$ and taking into account  Corollary
\ref{fourtris}.
\end{proof}

\medskip

\subsection{Regularity results for critical subsolutions}

We state and prove the main regularity results of this section.
They can be considered as a generalization to the network setting of the results in \cite{FathiSiconolfiC1}.\\

\begin{Theorem}\label{regolauno}  Any critical subsolution $u:\Gamma \to \R$ is of
class $C^1$ in $\A_\Gamma \setminus \VV$, and all such subsolutions
possess the same differential in $\A_\Gamma \setminus \VV$.
\end{Theorem}
\begin{proof} Let $u$ be a critical subsolution on $\Gamma$ and
$\ga=\Psi(e)$ an arc with $e \in \A_\XX^*$. According to Lemma
\ref{preregola}, formula \eqref{preregola00}
\[u(\ga(1))- u(\ga(0)) = \si_c(e),\]
therefore $u \circ \ga$ is the maximal subsolution taking the value
$u(\ga(0))$ at $s=0$ and, according to Proposition \ref{state}, has
the form
\[u(\ga(s))=  \int_0^s \si_c^+(t)\,dt,\]
where $\si_c^+$ is as in \eqref{sia} with $H_\ga$ in place of $H$
and $c$ in place of $a$.  We deduce that $s \mapsto u(\ga(s))$  is
of class $C^1$ for $t \in (0,1)$ and for any $x = \ga(t_0)$, with
$t_0\in (0,1)$,  the differential $D_\G u(x)$ is uniquely determined
among the elements of $T^*_\Gamma(x)$ by the condition
\[ (D_\G u(x),  \dot\ga(t_0)) = \frac d{dt} u(\ga(t))\big |_{t=t_0}=\si_c^+(t_0).\]
This concludes the proof.
\end{proof}

\medskip

Moreover:

\begin{Theorem}\label{regolaregola} For any  critical subsolution $w$ on $\XX$, there exists
 a critical subsolution $u$ on $\Gamma$, with $w=u$ on $\VV$,  which is of class $C^1$
in $\G \setminus \VV$.  There exists  in addition a critical
subsolution $v$ on $\Gamma$ of class $C^1(\G \setminus \VV)$
satisfying
\[ \HN(x, D_\G v(x)) < c \qquad\hbox{ for $x \in \G \setminus (\A_\G
\cup \VV)$}.\]
\end{Theorem}

\begin{proof} Let $w$ be a critical subsolution in $\XX$.  Given any arc  $\ga =
\Psi(e)$, we know,  see Proposition \ref{relationHJDEF}, that $
w(\ga(0))$ and $ w(\ga(1))$ satisfy the compatibility condition
\eqref{compa}, so that
\begin{equation}\label{regolaregola5}
    w(\ga(0)) + \int_0^1 \si_c^-(t) \, dt \leq w(\ga(1)) \leq  w(\ga(0)) + \int_0^1\si_c^+(t) \,
dt,
\end{equation}
where $\si_c^+$, $\si_c^-$ are defined as in \eqref{sia},
\eqref{siabis} with $H_\ga$, $c$ in place of $H$, $a$, respectively.
We can therefore find  $\la \in [0,1]$ with
 \begin{equation}\label{regolaregola6}
     w(\ga(1)) =  w(\ga(0)) + \int_0^1  \big [\la \,\si_c^-(t) +(1-\la)\, \si^+_c(t) \big ]\,
dt,
\end{equation}
and the function
\begin{equation}\label{regolaregola7}
   s \mapsto w(\ga(0)) + \int_0^s \big [ \la \,\si_c^-(t) +(1-\la)\, \si^+_c(t) \big ]\,
dt
\end{equation}
 is a  subsolution of class $C^1$ to $H_\ga=c$ in $(0,1)$ taking
the values $w(\ga(0))$, $w(\ga(1))$ at $s=0$ and $s=1$,
respectively. This shows the first part of the assertion.

As far as the second claim is concerned, we proceed by taking a critical subsolution $w$  satisfying
\eqref{preregola000}. This implies that strict inequalities prevail
in formula \eqref{regolaregola5}  whenever $\ga = \Psi(e)$ with $e$,
$-e$ not in $\A_\XX^*$. The $\la$ appearing in \eqref{regolaregola6}
can be consequently taken in $(0,1)$, so that the function defined
in \eqref{regolaregola7} is a strict subsolution to $H_\ga=c$. This
concludes the proof in the light of Remark \ref{regolo}.
\end{proof}

\smallskip

\begin{Remark} Notice that if we apply the  procedure
of first part of the previous result starting with a critical
solution rather than a critical subsolution, then the  property of being solution could be possibly false for the
regularized function.\\
\end{Remark}

\subsection{Representation formulae and uniqueness results on the network}
In this section, we want to provide representation formulae and
uniqueness results  with traces that are not necessarily defined on
vertices, but on a general subset of the network $\G$. To this aim,
we extend $S_a$, for $a \geq c$, from $\VV$ to the whole $\G$
defining a semidistance intrinsically related to $\HN$ and the level
$a$. This is basically the same object introduced in
\cite{CamilliSchieborn}. We do not develop here any further the
metric point of view, but just use it to establish an admissibility
condition for data assigned on subsets of $\G$, and provide
representation formulae.

Given a portion of arc $\ga \big |_{[s_1,s_2]}$, for $0\leq s_1 \leq
s_2 \leq 1$,  we define
\[\ell_a\left ( \ga \big |_{[s_1,s_2]} \right ) = \int_{s_1}^{s_2} {(\si_a^+)}^{\ga}(t)
\,dt,\] where ${(\si_a^+)}^{\ga}$ is  defined as in \eqref{sia}. We
get in particular, for the  whole arc, the relation
\begin{equation}\label{intrinsic0}
   \ell_a(\ga)= \si_a(\Psi^{-1}(\ga)) \qquad\hbox{for any $\ga \in
\EN$.}
\end{equation}
We define $\ell_a$ for  a curve on $\G$ given by a finite number of
concatenated arcs or portions of arcs  as the sum of the lengths of
the arcs or portion of arcs making  it up. We introduce the related
geodesic (semi)distance on $\G$ via
\begin{equation}\label{SGa}
   S^\G_a(x,y)= \min  \{ \ell_a(\xi) \mid \xi \;\hbox{union of concatenated arcs linking $x$ to
$y$}\}.
\end{equation}
 We deduce from the  results on $\si_a$  and
\eqref{intrinsic0} the following lemma.

\begin{Lemma} \hfill

\begin{itemize}
    \item [{i)}] If $x \neq y$ are in  $\VV$, then $S_a(x,y)=S^\G_a(x,y)$.
    \item  [{ii)}] If $\xi$ is a closed curve on $\G$, then $\ell_a(\xi) \geq
    0$.
\end{itemize}
\end{Lemma}

It is easy to check that the maximal subsolution $v$ to  \eqref{HJa}
vanishing at $y \in \G$ given in Theorem \ref{deniro} and
Proposition \ref{postdeniro} is
\[v(x) = S_a^\G(y,x) \qquad\hbox{for any  $a \geq c$, $x \in \G$.}\]
We derive, taking also into account  Proposition \ref{sa},  that for
a continuous function $u : \G \to \R$, the condition
\begin{equation}\label{ancora}
   u(x) - u(y) \leq S_a^\G(y,x) \qquad \hbox{for any pair $x$, $y$ in
$ \G'$}
\end{equation}
 is necessary and sufficient  for being subsolution to
\eqref{HJ}.  Given a function $g$ defined on a subset $\G'$ of $\G$,
we therefore introduce the following admissibility condition for
\eqref{HJa}
\begin{equation}\label{newcompa}
   g(x) - g(y) \leq S_a^\G(x,y) \qquad\hbox{for any $x$, $y$ in
   $\G'$.}
\end{equation}

\smallskip

We give in the next theorem a couple  of examples of uniqueness
results for solutions to \eqref{HJa}, and corresponding
representation formulae, one can obtain prescribing values on
subsets  not necessarily  contained in $\VV$. Further results are
reachable following the same line. Similar formulae, even if for
subsets of vertices and just in the supercritical case, have been
already obtained in \cite{CamilliSchieborn}.

\smallskip

\begin{Theorem} \label{thisistheend} Let $\G'$  be a closed  subset of
$\G$ and $g$ an admissible trace defined on it, in the sense of
\eqref{newcompa}. We set
\[v(x)= \min\{g(y) + S^\G_a(y,x) \mid y \in \G'\}.\]
\begin{itemize}
\item[(i)]{\sc Critical case:}
if $a =c$ and $\G' \subset \A_\G$ with
\begin{equation}\label{ancorabis}
    \G' \cap \ga([0,1]) \neq \emptyset \qquad\hbox{ for any $\ga$ with $\Psi^{-1}(\ga) \in
\A^*_\XX$,}
\end{equation}
then $v$ is the unique solution  in $\G$ to $\HN(x,Du)=c$ agreeing
with $g$ on $\G'$.
\item[(ii)] {\sc Supercritical case:}  If  $ a >c$,  then $v$ is  uniquely characterized by the properties of being  in $
C(\G,\R)$, being solution of \eqref{HJ} in $\G \setminus \G'$, and
agreeing with $g$ on $\G'$.
\end{itemize}
\end{Theorem}
\begin{proof}  The solution property of $v$ in both cases, in $\G$ and $\G \setminus \G'$ respectively, follows directly from
being  a subsolution  in $\G$, in force of  \eqref{ancora},   and
satisfying the subtangent test as minimum of solutions, in $\G$ and
$\G \setminus \G'$ respectively. In addition $v$   is the maximal
solution (in $\G$ or $\G \setminus \G'$) agreeing with $g$ on $\G'$
in force of Theorem  \ref{deniro}, Proposition \ref{postdeniro}, and
 the admissibility condition \eqref{newcompa}.

Now, assume $u$ to be another solution taking the value $g$ on
$\G'$, by adapting the backward procedure  explained in Proposition
\ref{precritico} and Theorem \ref{uniquenessset}, we construct, for
any $x \in \G \setminus \G'$, a curve $\xi$ made up by concatenated
arcs or portion of arcs starting at some point $y \in\G'$ and
ending at $x$ with
\[ u(x) = g(y) +\ell_a(\xi) \geq v(x).\]
In the critical case  condition \eqref{ancorabis} plays a crucial
role  for this. The maximality property of $v$ then implies that
equality must hold
 in the above formula. This ends the proof.
\end{proof}

\bigskip

\section{Summary of the Main Results} \label{summary}
In this final section we  summarize our results  for the Hamilton--Jacobi equation posed on the network.\\

\noindent {\bf Main Theorem.} {\it Let $\G$ be an embedded network
(finite, connected, possibly including loops and more arcs
connecting two vertices)  and let $\XX=(\VV,\EE)$ be the underlying
abstract graph. Let us consider a Hamiltonian
$\HN=\{H_{\gamma}\}_{\ga \in \EN}$ on the network, satisfying
conditions {\bf (H$\ga$1)--(H$\ga$4)} for any $\ga\in\EN$ and let
$a_0$ denote the value defined in \eqref{a0}. Then:\\

\begin{itemize}
\item[I.] {\sc Global Solutions:}\\
\begin{itemize}
\item[(i)] {\sc (Existence)}  There exists a unique value $c=c(\HN)\geq a_0$ -- called {\em Ma\~n\'e critical value}
 -- for which the equation $\HN(x,Du)= c$
admits global solutions. In particular, these solutions are Lipschitz continuous
on $\G$.\\

\item[(ii)] {\sc (Uniqueness)}
Let $\A_\XX=\A_\XX(\HN) \subseteq \VV$ be the {\em (projected) Aubry
set} associated to $\HN$ and let $S_c: \VV\times\VV \longrightarrow
\R$ be the function defined in \eqref{Saa}. Then, given any {\em
admissible trace} $g$ on $\A_\XX$, i.e., a function $g:\A_\XX
\longrightarrow \R$ such that for every $x,y\in A_\XX$
\begin{equation*}
   g(x) -g(y) \leq S_c(y,x),
\end{equation*}
there exists a unique global solution $u\in C(\G,\R)$ to $\HN(x,Du)=c$ agreeing  with $g$ on $\A_\XX$.
Conversely, for any  solution $u$ to $\HN(x,Du)=c$, the function  $g=u_{|\A_\XX}$ gives rise to admissible trace on $\A_\XX$.\\

\item[(iii)] {\sc (Hopf--Lax type formula 1)}
Let $g:\A_\XX \longrightarrow \R$ be an admissible trace and  $u\in C(\G,\R)$  the corresponding solution to $\HN(x,Du)=c$.
Then, on the support of any arc $\ga \in \EN$, $u$ is given by
\[ u(\ga(s)) =  \min \{ {\bf A} , \, {\bf B}\},\] where
\begin{eqnarray*}
  {\bf A} &:=& \min \{g(y) + S_c(y,\ga(0)) \mid y \in \A_\XX\}  +
\int_0^s \si_c^+(t)\,dt \\
 {\bf B} &:=& \min \{g(y) + S_c(y,\ga(1)) \mid y \in
\A_\XX\} - \int_s^1 \si_c^-(t)\,dt ,
\end{eqnarray*}
with $s \in [0,1]$ and $\si_c^+$, $\si_c^-$  defined as in
\eqref{sia}, \eqref{siabis} with $H_\ga$ in place of $H$.\\
\\

\item[(iv)]{\sc (Hopf--Lax type formula 2):}
Let $\G'$  be a closed  subset of $\G$ with
\[ \G' \cap \ga([0,1]) \neq \emptyset \qquad\hbox{ for any $\ga$ with $\Psi^{-1}(\ga) \in
\A^*_\XX$.}\] For  any {\em admissible trace} $g$ on $\G'$, in the
sense of \eqref{newcompa} with $c$ in place of $a$, there exists a
unique solution $u\in C(\G,\R)$ to $\HN(x,Du)=c$ agreeing with $g$
on $\G'$, which is given by
\[u(x)= \min\{g(y) + S^\G_c(y,x) \mid y \in \G'\},\]
where $S^\G_c(\cdot,\cdot)$ denotes the intrinsic (semi)distance defined in \eqref{SGa}.\\

\end{itemize}

\item[II.] {\sc Subsolutions:}\\

\begin{itemize}
\item[(i)] {\sc (Maximal subsolutions)} For $a \geq c$, $y \in
\G$, the maximal subsolution to \eqref{HJ} taking an assigned value
at $y$ is solution in $\G \setminus \{y\}$. \\

\item[(ii)] {\sc (PDE characterization of the Aubry set)}
Let $\A_\Gamma=\A_\Gamma(\HN) \subset \G$ be the Aubry set on the
network, as defined in \eqref{Aubrynetwork}.
The maximal subsolution to $(\mathcal{H}J\mathit{c})$ taking a given value at a point $y \in \G$  is   a
critical solution on the whole network if and only if $y \in \A_\G$.
\\

\item[(iii)] {\sc (Regularity of critical subsolutions)}
Any subsolution $v:\G
\to \R$ to $\HN(x,Du)=c$ is of class $C^1(\G \setminus \VV)$  and
they all possess the same differential on $\A_\G \setminus \VV$.
More specifically, if $x_0\in \A_\G$ and $x_0 = \ga(s_0)$, for some
$\ga\in\EN$ and $s_0 \in (0,1)$, then its differential at $x_0$ is
uniquely determined by the relation
\[ ( D_\G v(x_0), \dot\ga(s_0) )= \si_c^+(s_0),\]
where $\si_c^+$ was defined  in \eqref{sia}, and therefore
\[v(\ga(s)) = v(\ga(0)) + \int_0^s \si^+_c(t) \, dt  \qquad \mbox{for any}\; s\in [0,1].\]
We infer from this that any pair of critical subsolutions differs by a
constant on the support of $\ga$.\\

\item[(iv)] {\sc (Existence of $C^1$ critical subsolutions)}
Given a  function $g: \VV \longrightarrow \R$ such that
\begin{equation*}
   g(x) -g(y) \leq S_c(y,x) \qquad \forall\; x,y\in \VV,
\end{equation*}
there exists a critical subsolution  $v$ on $\Gamma$, with $v=g$ on $\VV$,  which is of class $C^1$
on  $\Gamma \setminus \VV$. \\
In addition, there exists   a critical subsolution $v$ of class
$C^1(\G \setminus \VV)$  satisfying
$$H_\ga(s,Dv(\ga(s)))< c$$
for all  $s \in (0,1)$ and $\ga\in\EN$ with $\ga((0,1)) \cap \A_\Gamma = \emptyset$.
\\

\item[(v)] {\sc (Hopf--Lax formula for maximal supercritical subsolutions 1)} Let $a > c$
and  $\VV' \subset \VV$. For any $g: \VV' \longrightarrow \R$
satisfying
\begin{equation*}
   g(x) -g(y) \leq S_a(y,x) \qquad \forall\; x,y\in \VV',
\end{equation*}
where $S_a(\cdot,\cdot)$ was defined in \eqref{Saa}, there exists a
unique solution $u$ to $\HN(x,Du)=a$ in $\G\setminus \VV'$ agreeing
with $g$ on $\VV'$;  in addition, $u$ is also a subsolution to
$\HN(x,Du)=a$ on the whole $\G$. In particular, on the support of
any arc $\ga \in \EN$, $u$ is given by:
\[ u(\ga(s)) =  \min \{ {\bf C} , \, {\bf D}\},\] where
\begin{eqnarray*}
  {\bf C} &:=& \widetilde{g}(\ga(0))  +
\int_0^s \si_a^+(t)\,dt \\
 {\bf D} &:=& \widetilde{g}(\ga(1)) - \int_s^1 \si_a^-(t)\,dt \\
 \widetilde{g}(x) &:=& \left \{\begin{array}{cc}
    g(x) & \quad\hbox{if $x \in \VV'$} \\
   \min  \{g(y) + S_a(y,x) \mid y \in \VV'\} & \quad\hbox{if $x \not\in \VV'$} \\
    \end{array} \right.
\end{eqnarray*}
with $s \in [0,1]$ and $\si_a^+$, $\si_a^-$  defined as in
\eqref{sia}, \eqref{siabis}.
\\

\item[(vi)] {\sc (Hopf--Lax formula for  maximal supercritical subsolutions 2)} Let $ a >c$ and $\G'$ be a closed subset of
$\G$. Let $g$ be an  admissible trace  on $\G'$, in the sense of
\eqref{newcompa},   then there exists a unique solution $u \in
C(\G,\R)$ to \eqref{HJ} {on $\G\setminus \G'$} agreeing with $g$ on $\G'$, which is given by
\[u(x)= \min\{g(y) + S^\G_a(y,x) \mid y \in \G'\},\]
where $S^\G_a(\cdot,\cdot)$ denotes the intrinsic (semi)distance defined in \eqref{SGa}.

\end{itemize}

\end{itemize}

}

\begin{proof} \hspace{10 pt}
\begin{itemize}
\item[I.] (i) Existence follows from Theorem \ref{critico} and Proposition \ref{relationHJDEF}; Lipschitz continuity from Proposition
\ref{Lipcont}.\\
(ii) This part is obtained by combining  Proposition
\ref{relationHJDEF} and Theorem \ref{uniquenessset}\\
(iii) This representation formula was proven in Proposition
\ref{statextra}.\\
(iv) See Theorem \ref{thisistheend} (i).\\

\item[II.]
(i) See Theorem \ref{deniro}.\\
(ii) See Proposition \ref{postdeniro}.\\
(iii) See Theorem \ref{regolauno}.\\
(iv) See Theorem \ref{regolaregola}.\\
(v) These results are obtained by combining Proposition
 \ref{prop5.8}, Proposition \ref{prop5.11}, Theorem \ref{prop5.23} and using the
representation formula in Proposition \ref{statextra}.\\
(vi) See Theorem \ref{thisistheend} (ii).\\
\end{itemize}
\end{proof}

\begin{appendix}

\section{}\label{proofs}

\medskip

\begin{proof} {\bf (Proposition \ref{Lipcont}) }\;\;
 Taking into account that for any $\gamma \in \EN$ (which is a
finite set) $w\circ \ga$ is Lipschitz--continuous in $[0,1]$, thanks
to the coercivity condition {\bf (H$\ga$2)}, we deduce that there
exists $L>0$ such that, for any given subsolution $w$
\begin{equation}\label{lipcont2}
   |w(\ga(s_2)) -w(\ga(s_1))| \leq L \; \ell \left ( \ga  \big |_{[s_1,s_2]} \right ) \qquad\hbox{for all
$\ga \in \EN,$ and $s_1\leq s_2 \in[0,1]$;}
\end{equation}
hereafter $\ell$ indicates the Euclidean length of curves in $\R^N$.

 We proceed by considering
$x$ and $y$ in $\Gamma$ and a finite sequence of concatenated arcs
$\ga_1, \cdots \ga_M$, for some index $M$, that realize the geodesic
distance $d_\G(x,y)$. More specifically, we assume  that $x =
\ga_1(t_x)$, $y=\ga_M(t_y)$ with $t_x$, $t_y$ in $[0,1]$ and that
\[d_\Gamma(x,y) \; =\; \ell \left ( \ga_1 \big |_{[t_x,1]} \right ) +
\sum_{i=2}^{M-1} \ell(\ga_i) + \ell \left ( \ga_M \big |_{[0,t_y]}
\right ).\]
In the remainder of the proof we assume that $M>2$ in order to ease the notation (the other cases can be treated analogously).  \\

We deduce from \eqref{lipcont2} that
\begin{eqnarray*}
  |w(y)-w(x)| &\leq& |w(\ga_1(1))-w_1(\ga_1(t_x))| \\ &+&\sum_{i=2}^{M-1}
|w(\ga_i(1))-w(\ga_i(0))| + |w(\ga_M(t_y))-w_1(\ga_M(0))| \\
   &\leq&  L\, \left[ \ell \left ( \ga_1 \big |_{[t_x,1]} \right ) +
\sum_{i=2}^{M-1} \ell(\ga_i) + \ell \left ( \ga_M \big |_{[0,t_y]}
\right )  \right] \\
&=& L \, d_\Gamma(x,y).
\end{eqnarray*}

This concludes the proof.
\end{proof}

\bigskip

\begin{proof} {\bf (Proposition \ref{statextra})} \;\; We denote by $w$ the function appearing in the statement.
If $a =a_\ga$, the assertion comes from \eqref{siasia} and
Proposition \ref{equnique}. Instead, if $a > a_\ga$,  the function
$w$ is an a.e. subsolution, being the minimum of two  $C^1$
(sub)solutions. Using a basic property in viscosity solutions
theory,  it is also a supersolution, as minimum of supersolutions.
Moreover, $w(0)=\al$, $w(1)= \be$ hold thanks to \eqref{compa}.

Finally,  the function $s \longmapsto \int_0^s \si^+_{a_\ga}$ is a
strict subsolution to \eqref{HJg}, and this implies by an argument
going back to \cite{Ishii}  that the Dirichlet problem with
admissible data $\al$, $\be$ is uniquely solved.
\end{proof}

\bigskip

\begin{proof} {\bf (Proposition \ref{state})} \;\;
If $a =a_\ga$, then, as already pointed out in Proposition
\ref{equnique}, the solution is unique up to additive constants,
hence it is automatically given by \eqref{state1} once the value
$w(0)$ is assigned.

Therefore, from now on we can  assume that $a
>a_\ga$. By Proposition
\ref{statextra}
\[w(s) = \min \left \{ w(0) +\int_0^s \si_a^+(t) \, dt , \, w(1) - \int_s^1 \si_a^-(t) \, dt \right \}
\qquad\hbox{for any $s$.}\] We  claim that if
\begin{equation}\label{state2}
    w(s_0) =  w(1) - \int_{s_0}^1 \si_a^-(t) \, dt
\end{equation}
for some $s_0 \in (0,1)$, then
\[ w(s) =  w(1) - \int_s^1 \si_a^-(t) \, dt \qquad\hbox{for any $s \in (s_0,1]$.} \]
Assume by contradiction that there exists  $s_1 > s_0$ such that
\[ w(0) + \int_0^{s_1} \si_a^+(t) \, dt = w(0)  + \int_0^{s_0}\si_a^+(t) \,dt+ \int_{s_0}^{s_1} \si_a^+(t) \, dt <
w(1)- \int_{s_1}^1 \si^-(t) \,dt;\] this implies that
\begin{equation}\label{state3}
   w(0) +\int_0^{s_0} \si_a^+(t) \,dt  < w(1) - \int_{s_1}^1 \si_a^-(t)\,dt  -
   \int_{s_0}^{s_1}
   \si_a^+(t) \,dt.
\end{equation}
It is apparent  that
\[ \int_{s_0}^{s_1} \si_a^+(t) \, dt > \int_{s_0}^{s_1} \si_a^-(t) \,dt \]
and we can consequently  deduce from \eqref{state3} that
\[ w(0)  +\int_0^{s_0} \si_a^+(t) \,dt  < w(1)- \int_{s_1}^1 \si_a^-(t)\,dt - \int_{s_0}^{s_1} \si_a^-(t)\,dt =
w(1)- \int_{s_0}^1 \si^-(t)\,dt,\] in contrast with \eqref{state2}.
We assume, for purposes of contradiction, that \eqref{state2} holds
true  for some  $s_0\in (0,1)$. Since $a
> a_\ga$, we can take $p_0$ with $H(1,p_0) < a$. If $w$ is not of
the form \eqref{state1}, then, owing to the previous claim, we can
fix $s_0$ in such a way that
\begin{eqnarray*}
 w(s) = w(1) - \int_s^1\si^-_a(t) \,dt  \qquad {\rm and} \qquad
 H(s,p_0) < a
\end{eqnarray*}
for $s \in [s_0,1]$. This implies
\[\varphi(s):= w(1) + p_0 (s-1) \leq w(1) - \int_s^1\si^-_a(t) \,dt= w(s),\]
for $s \in [s_0,1]$, and consequently $\varphi$ is a constrained
subtangent to $w$ at $1$ with
\[ H(1,\varphi'(1))= H(1,p_0) <1,\]
contradicting \eqref{state0}. We deduce that $w$ is of the form
\eqref{state1}  showing  the first part of the assertion.

Conversely, if $w$ is of the form \eqref{state1}, then it is of
class $C^1$  in $(0,1)$ with $w'(s)= \si^+_a(s)$. Consequently, any
constrained subtangent $\varphi$ at $t =1$ must  satisfy
\[w(1) - \int_s^1 \varphi' \, dt = \varphi (s) \leq w(s)= w(1)- \int_s^1 \si^+_a \, dt\]
for $s$ sufficiently close to $1$. This  implies
\[ \int_s^1 \varphi' \, dt \geq \int_s^1 \si^+_a \, dt\]
and shows the existence of a sequence $s_n$ contained in $(0,1)$ and
converging to $1$ as $n$ goes to infinity, with
$\varphi'(s_n) \geq \si^+_a(s_n).$
Passing to the limit as $n$ goes to infinity, we get
$\varphi'(1) \geq \si^+_a(1).$
We deduce from this the inequality \eqref{state0} and conclude the
proof.
\end{proof}

\bigskip

\begin{proof} {\bf (Proposition \ref{aubryall})}\;\;  If $a=c_\ga=a_\ga$ then the integrals in \eqref{aubryall1} coincide in force of \eqref{siasia},
then they must both vanish, and this shows the assertion. Assume now
that $c_\ga > a_\ga$   and also assume for purposes of contradiction
that strict inequalities prevail instead in \eqref{aubryall1}. Then,
we can  find $\la \in (0,1)$ with
\[\int_0^1 \big [ \la \, \si_{c_\ga}^+ (t) + (1-\la) \, \si_{c_\ga}^-(t) \big ]
\,dt =0.\] Taking into account that $\si_{c_\ga}^+(t) >
\si_{c_\ga}^-(t)$ for any $t$,  this implies that
\[s \mapsto \int_0^s \big [ \la \, \si_{c_\ga}^+ (t) + (1-\la) \, \si_{c_\ga}^- (t) \big
] \,dt\] is a strict periodic subsolution to $H= c_\ga$.  This is
impossible by the very definition of $c_\ga$.
\end{proof}

\bigskip

\begin{proof}{\bf (Corollary \ref{fourtris})}\;\; The unique point to check is that the values $\al+\be$ at $s=0$
and $\al$ at $s= s_0$ are admissible, in the sense of \eqref{compa},
for \eqref{HJg} in $(0,s_0)$, and the same holds true in $(s_0,1)$
for the value $\al$ at $s=s_0$ and $\al+\be$ at $s =1$. The argument
is the same for the two subintervals. We therefore focus on
$[s_0,1]$.

If $u(1)-u(s_0)=\be = \int_{s_0}^1 \si^+a(t) \,dt$ the compatibility
property is immediate  and the solution in $(s_0,1)$ is given by
\eqref{fourtris2}, as asserted in item ii) of the statement.  Let us
instead assume
\begin{equation}\label{fourtris3}
   u(1)-u(s_0) = \be=- \int_0^{s_0} \si^-a(t) \,dt < \int_{s_0}^1 \si^+a(t)
   \,dt.
\end{equation}
We have by Lemma \ref{six} $\int_0^1 \si_a^-(t) \, dt \leq 0$ and
consequently
\[u(1)-u(s_0) \geq \int_{s_0}^1 \si_a^-(t) \, dt.\]
The last inequality plus \eqref{fourtris} shows the claimed
admissibility property. This concludes the proof.
\end{proof}

\bigskip

\begin{proof} {\bf (Proposition \ref{unico} )} \;\; Let $w$ be a solution to \eqref{HJ} with trace $u$ on $\VV$. We
know by the very definition of solution that given any arc $\ga$,
then $w \circ \ga$ is a solution to $H_\ga = a$ in $(0,1)$ taking
the values $u(\ga(0))$ and $u(\ga(1))$ at $0$ and $1$, respectively.
This implies  that such boundary values are admissible with respect
to $H_\ga$, in the sense of formula \eqref{compa} with $H_\ga$ in
place of $H$. By the uniqueness property showcased in Proposition
\ref{statextra}, the values of $w$ on the support of $\gamma$ are
therefore uniquely determined by $u(\ga(0))$, $u(\ga(1))$ and
$H_\ga$. Since the arc $\gamma$ has been arbitrarily chosen, we can
hence conclude the asserted uniqueness.
\end{proof}

\end{appendix}

\bigskip

%%%%%%%%%%%%% BIBLIOGRAPHY %%%%%

\end{document}